\documentclass[10pt]{amsart}
\usepackage{latexsym,amssymb,amscd}
\usepackage{graphicx}
\usepackage{epstopdf}
\usepackage{amsmath}

\def\NZQ{\Bbb}               % the font for N,Z,Q,R,C

\def\ZZ{{\NZQ Z}}

\def\B'c{{\mathcal{B'}}}
\def\U'c{{\mathcal{U'}}}

\def\opn#1#2{\def#1{\operatorname{#2}}} % to make operators
\opn\chara{char}
\opn\length{\ell}
\opn\projdim{proj\,dim}
\opn\injdim{inj\,dim}
\opn\ini{in}
\opn\rank{rank}
\opn\depth{depth}
\opn\sdepth{sdepth}
\opn\height{ht}
\opn\embdim{emb\,dim}
\opn\codim{codim}

\opn\Tr{Tr}
\opn\bigrank{big\,rank}
\opn\superheight{superheight}\opn\lcm{lcm}
\opn\trdeg{tr\,deg}%
\opn\reg{reg}
\opn\lreg{lreg}
\opn\set{set}
\opn\supp{Supp}
\opn\shad{Shad}
\opn\div{div}
\opn\Div{Div}
\opn\cl{cl}
\opn\Cl{Cl}
\opn\Spec{Spec}
\opn\Supp{Supp}
\opn\supp{supp}
\opn\Sing{Sing}
\opn\Ass{Ass}
\opn\Min{Min}
\opn\size{size}
\opn\bigsize{bigsize}
\opn\lex{lex}
\opn\d{d}
\opn\diam{diam}
\opn\Ann{Ann}
\opn\Rad{Rad}
\opn\Soc{Soc}
\opn\Ker{Ker}
\opn\Coker{Coker}
\opn\Im{Im}
\opn\Hom{Hom}
\opn\Tor{Tor}
\opn\Ext{Ext}
\opn\End{End}
\opn\Aut{Aut}
\opn\id{id}

\opn\nat{nat}
\opn\GL{GL}
\opn\SL{SL}
\opn\mod{mod}
\opn\ord{ord}
\opn\aff{aff}
\opn\con{conv}
\opn\relint{relint}
\opn\st{st}
\opn\lk{lk}
\opn\cn{cn}
\opn\core{core}
\opn\vol{vol}
\opn\gr{gr}

\def\pot#1#2{#1[\kern-0.28ex[#2]\kern-0.28ex]}

\opn\dirlim{\underrightarrow{\lim}}
\opn\invlim{\underleftarrow{\lim}}
\def\pnt{{\raise0.5mm\hbox{\large\bf.}}}

\def\Implies{\ifmmode\Longrightarrow \else
     \unskip${}\Longrightarrow{}$\ignorespaces\fi}
\def\implies{\ifmmode\Rightarrow \else
     \unskip${}\Rightarrow{}$\ignorespaces\fi}
\def\iff{\ifmmode\Longleftrightarrow \else
     \unskip${}\Longleftrightarrow{}$\ignorespaces\fi}

\let\:=\colon
\newtheorem{Theorem}{Theorem}[section]
\newtheorem{Lemma}[Theorem]{Lemma}
\newtheorem{Corollary}[Theorem]{Corollary}
\newtheorem{Proposition}[Theorem]{Proposition}
\newtheorem{Remark}[Theorem]{Remark}

\newtheorem{Example}[Theorem]{Example}

\newtheorem{Definition}[Theorem]{Definition}

\newtheorem{Conjecture}[Theorem]{Conjecture}
\newtheorem{Question}[Theorem]{Question}
\let\epsilon=\varepsilon
\let\phi=\varphi
\let\kappa=\varkappa
\numberwithin{equation}{section}

\title{Depth and Stanley depth of the edge ideals of the strong product of some graphs}
\author[Zahid Iqbal]{Zahid Iqbal}
\address{Zahid Iqbal, Muhammad Ishaq, School of Natural Sciences, National University of Sciences and Technology Islamabad, Sector H-12, Islamabad Pakistan.}
\email{786zahidwarraich@gmail.com, ishaq$\_\,$maths@yahoo.com}
\author[Muhammad Ishaq]{Muhammad Ishaq}
\author[Muhammad Ahsan Binyamin]{Muhammad Ahsan Binyamin}
\address{Muhammad Ahsan Binyamin, Department of Mathematics, Government College University Faisalabad Pakistan.}
\email{ahsanbanyamin@gmail.com}

\begin{document}
\maketitle
\begin{abstract}
In this paper we study depth and Stanley depth of the edge ideals and quotient rings of the edge ideals, associated to classes of graphs obtained by taking the strong product of two graphs. We consider the cases when either both graphs are arbitrary paths or one is an arbitrary path and the other is an arbitrary cycle. We give exact formulae for values of depth and Stanley depth for some subclasses. We also give some sharp upper bounds for depth and Stanley depth in the general cases.
\end{abstract}
\section{Introduction}
\noindent Let $S := K[x_{1}, \ldots, x_{n}]$ be the polynomial ring over field $K$. Let $M$ be a finitely generated
$\mathbb{Z}^{n}$-graded $S$-module. A Stanley decomposition of $M$ is a
presentation of $K$-vector space $M$ as a finite direct sum
$\mathcal{D}: M = \bigoplus_{i = 1}^{r}w_{i}K[A_{i}],$ where $w_{i}\in M$, $A_{i}\in \{x_{1}, \ldots, x_{n}\}$ such that
$w_{i}K[A_{i}]$ denotes the $K$-subspace of $M$, which is generated
by all elements $w_iu$, where $u$ is a monomial in $K[A_{i}]$. The
$\mathbb{Z}^{n}$-graded $K$-subspace $w_{i}K[A_{i}]\subset M$ is
called a Stanley space of dimension $|A_{i}|$, if $w_{i}K[A_{i}]$ is
a free $K[A_{i}]$-module, where $|A_{i}|$ denotes the number of
indeterminates of $A_{i}$. Define
$\sdepth(\mathcal{D}) = \min\{|A_{i}|: i = 1, \ldots, r\},$ and
$\sdepth(M) = \max\{\sdepth(\mathcal{D}): \text{$\mathcal{D}$~is~a~Stanley~decomposition~of
~$M$}\}.$ The number $\sdepth(\mathcal{D})$ is called the Stanley
depth of decomposition $\mathcal{D}$ and $\sdepth(M)$ is called the Stanley depth
of $M$. Stanley Conjectured in \cite{RP} that $\sdepth(M)\geq \depth(M)$ for any $\ZZ^n$-graded $S$-module $M$. This conjecture was disproved by Duval et al. \cite{DG}.

Let $I\subset J\subset S$ be monomial ideals, Herzog et al. \cite{HVZ} showed that the invariant Stanley depth of $J/I$ is combinatorial in nature. The strange thing about Stanley depth is that it shares some properties and bounds with homological invariant $\depth$ see (\cite{HVZ,MI1,AR1,PFY}). Until now mathematicians are not too much familiar with Stanley depth as it is hard to compute, for computation and some known results we refer the readers to (\cite{BH,MI,MI2,KS,PFY}). Let $P_n$ and $C_n$ represent path and cycle respectively on $n$ vertices and $\boxtimes$ represent the strong product of two graphs. The aim of this paper is to study depth and Stanley depth of the edge ideals and quotient ring of the edge ideals associated to classes of graphs $\mathcal{H}:=\{P_n\boxtimes P_m:n,m\geq 1\}$ and $\mathcal{K}:=\{C_n\boxtimes P_m:n\geq3, m\geq1\}$. In section 3 we compute depth and Stanley depth of quotient ring of edge ideals associated to some subclasses of $\mathcal{H}$ and $\mathcal{K}$.

For the monomial ideal $I\subset S$ it is well known that $\depth(I)=\depth(S/I)$+1, this means that once you know about $\depth(S/I)$ then you also know about $\depth(I)$ and vice versa. Where as for Stanley depth this is not the case, we have examples where $\sdepth(I)>\sdepth(S/I)$ but till now no example is known where $\sdepth(I)<\sdepth(S/I)$. Looking at the behavior of $\sdepth(S/I)$ and $\sdepth(I)$ it seems that the latter inequality is false. In a recent survey on Stanley depth, Herzog conjectured the following inequality.
\begin{Conjecture}\cite{HS}\label{C1}
Let $I\subset S$ be a monomial ideal then $\sdepth(I)\geq \sdepth(S/I).$
\end{Conjecture}
In section 4 of this paper we confirm the above conjecture for the edge ideals associated to some subclasses of $\mathcal{H}$ and $\mathcal{K}$. For a recent work on the above conjecture we refer the reader to \cite{KY}. In section 5 we give sharp upper bounds for depth and Stanley depth of quotient ring of the edge ideals associated to $\mathcal{H}$ and $\mathcal{K}$. In the same section we also propose some open questions. We gratefully acknowledge the use of the computer algebra system CoCoA (\cite{COC}) for our experiments.

\section{Definitions and notation}
In this section we review some standard terminologies and notations from graph theory and algebra. For more details one may consult \cite{HIK,RHV}.
Let $G:=(V(G),E(G))$ be a graph with vertex set $V(G):=\{x_1,x_2,\dots,x_n\}$ and edge set $E(G)$. The edge ideal $I(G)$ associated to $G$ is the square free monomial ideal of $S$, that is $I(G)=(x_{i}x_{j} : \{x_i, x_j\}\in E(G)).$
A graph $G$ on $n\geq 2$ vertices is called a path on $n$ vertices if $E(G)=\{\{x_i,x_{i+1}\}:i=1,2\dots,n-1\}$. We denote a path on $n$ vertices by $P_n$. A graph $G$ on $n\geq 3$ vertices is called a cycle if $E(G)=\{\{x_i,x_{i+1}\}:i=1,2,\dots,n-1\}\cup\{\{x_1,x_n\}\}.$ A cycle on $n$ vertices is denoted by $C_n$. For vertices $x_i$ and $x_j$ of a graph $G$, the length of a shortest path from $x_i$ to $x_j$ is called the distance between $x_i$ and $x_j$ denoted by $\d_G(x_i, x_j)$. If no such path exists between $x_i$ and $x_j$, then $d_{G}(x_i,x_j)=\infty$. The diameter of a connected graph $G$ is $\diam(G):=\max\{\d_G(x_i,x_j):x_i,x_j\in V(G)\}$.
\begin{Definition}[\cite{HIK}]
The strong product $G_1 \boxtimes G_2$ of graphs $G_1$ and $G_2$ is a graph, with $V(G_1 \boxtimes G_2)=V(G_1)\times V (G_2)$ $($the cartesian product of sets$)$, and for $(v_1, u_1),(v_2, u_2)\in V(G_1\boxtimes G_2)$, $(v_1, u_1)(v_2, u_2)\in E(G_1 \boxtimes G_2)$, whenever
\begin{itemize}
  \item $v_1v_2 \in E(G_1)$ and $u_1 = u_2$ or
  \item $v_1 = v_2$ and $u_1u_2\in E(G_2)$ or
  \item $v_1v_2 \in E(G_1)$ and $u_1u_2\in E(G_2)$.
\end{itemize}
\end{Definition}
Let $P_1$ denotes the null graph on one vertex that is $V(P_1):=\{x_1\}$ and $E(P_1):=\emptyset$. Let $\mathcal{P}_{n,m}:=P_{n}\boxtimes P_{m}\cong P_{m}\boxtimes P_{n}$, if $n=m=1$, then $\mathcal{P}_{1,1}\cong P_1$, this trivial case is excluded. For $n\geq 3$ and $m\geq 1$, $\mathcal{C}_{n,m}:=C_{n}\boxtimes P_{m}\cong P_m\boxtimes C_n$.
\begin{Remark}
{\em
$|V(\mathcal{P}_{n,m})|=nm$, $|E(\mathcal{P}_{n,m})|=4(n-1)(m-1)+(n-1)+(m-1)$,
$|V(\mathcal{C}_{n,m})|=nm$ and $|E(\mathcal{C}_{n,m})|=|E(\mathcal{P}_{n,m})|+3(m-1)+1$.
}
\end{Remark}
%\newpage
Since both graphs $\mathcal{P}_{n,m}$ and $\mathcal{C}_{n,m}$ are on $nm$ vertices, for the sake of convenience we label the vertices of $\mathcal{P}_{n,m}$ and $\mathcal{C}_{n,m}$ by using $m$ sets of variables $\{x_{1j}, x_{2j}, \dots , x_{nj}\}$ where $1\leq j\leq m.$ We set $S_{n,m}:=K[\cup_{j=1}^m\{x_{1j}, x_{2j},\dots,x_{nj}\}]$. For examples of $\mathcal{P}_{n,m}$ and $\mathcal{C}_{n,m}$ see Fig \ref{A}.
\begin{Remark}
\begin{figure}[h!]
\centering
%\vspace{.002cm}
  %\Requires \usepackage{graphicx}
 \includegraphics[width=8cm]{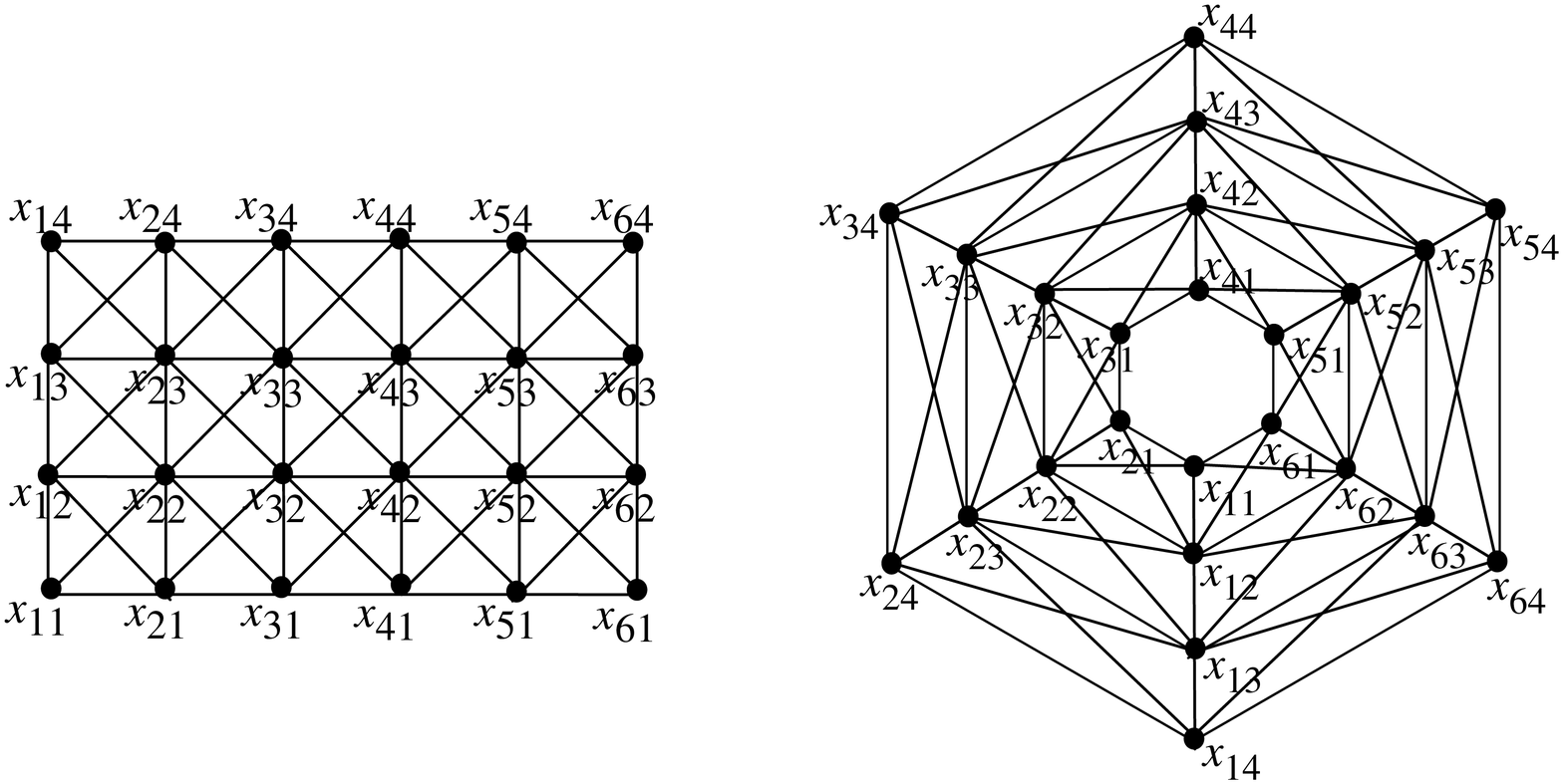}\\
 \caption {From left to right; $\mathcal{P}_{6,4}$ and $\mathcal{C}_{6,4}$.}\label{A}
\end{figure}

{\em Let $\mathcal{G}(I)$ denotes the unique minimal set of monomial generators of the monomial ideal $I$.
\begin{enumerate}
\item  For positive integers $m,n$ such that $m$ and $n$ are not equal to $1$ simultaneously, the minimal set of monomial generators of the edge ideal of $\mathcal{P}_{n,m}$ is given as:
{\small\begin{multline*}
$$\mathcal{G}(I(\mathcal{P}_{n,m}))=\cup^{n-1}_{i=1}\big\{\cup^{m-1}_{j=1}\{x_{ij}x_{i(j+1)},x_{ij}x_{(i+1)(j+1)},x_{ij}x_{(i+1)j},x_{(i+1)j}x_{i(j+1)}
,x_{nj}x_{n(j+1)}\},x_{im}x_{(i+1)m}\big\}.$$
\end{multline*}}
\item For $n\geq 3$, $m\geq 1$, the minimal set of monomial generators for $I(\mathcal{C}_{n,m})$ is:
{\small\begin{eqnarray*}
\mathcal{G}(I(\mathcal{C}_{n,m}))=\mathcal{G}(I(\mathcal{P}_{n,m}))\cup\big\{\cup^{m-1}_{j=1}\{x_{1j}x_{n(j+1)},x_{1j}x_{nj},x_{1(j+1)}x_{nj}\},
x_{1m}x_{nm}\big\}.
\end{eqnarray*}}
\item $\mathcal{P}_{n,1}\cong{P}_n$ and $\mathcal{C}_{n,1}\cong{C}_n$.
\item For $n,m\geq 1$, $\mathcal{P}_{n,m}\cong\mathcal{P}_{m,n}$, so without loss of generality the strong product of two paths can be represented as $\mathcal{P}_{n,m}$ with $m\leq n$. Thus in some proofs by induction on $n$, whenever we are reduced to the case where we have $\mathcal{P}_{n',m}$ with $n'<m$, in that case after a suitable relabeling of vertices we have $\mathcal{P}_{n',m}\cong \mathcal{P}_{m,n'}$. Therefore, we can simply replace $I(\mathcal{P}_{n',m})$ by $I(\mathcal{P}_{m,n'})$ and $S_{n',m}/I(\mathcal{P}_{n',m})$ by $S_{m,n'}/I(\mathcal{P}_{m,n'})$.
\end{enumerate}}
\end{Remark}
Now we recall some known results that are heavily used in this paper.
\begin{Lemma}\label{le01}
(Depth Lemma) If $0\rightarrow U \rightarrow M \rightarrow N \rightarrow 0$ is a short exact sequence of modules over a
local ring $S$, or a Noetherian graded ring with $S_0$ local, then
\begin{enumerate}
\item $\depth (M) \geq \min\{\depth(N), \depth(U)\}$.
\item $\depth (U) \geq \min\{\depth(M), \depth(N) + 1\}$.
\item $\depth (N) \geq \min\{\depth(U)- 1, \depth(M)\}$.
\end{enumerate}
\end{Lemma}
\begin{Lemma}[{\cite[Lemma 2.2]{AR1}}]\label{le1}
Let $0\rightarrow U\rightarrow V\rightarrow W\rightarrow 0$ be a short exact sequence of $\ZZ^{n}$-graded $S$-modules. Then
\,\,$\sdepth(V)\geq\min\{\sdepth(U),\sdepth(W)\}.$
\end{Lemma}
\begin{Lemma}[{\cite[Lemma 3.6]{HVZ}}]\label{111}
Let $I\subset S$ be a monomial ideal and $\bar{S}=S[x_{n+1},x_{n+2},\dots,x_{n+r}]$ be a polynomial ring in $n+r$ variables then $\depth(\bar{S}/I\bar{S})=\depth(S/IS)+r \text{\,\,\,\,\,and\,\,\,\,\,} \sdepth(\bar{S}/I\bar{S})=\sdepth(S/IS)+r.$
\end{Lemma}
\begin{Theorem}[{\cite[Theorem 2.3]{O}}]\label{oka}
Let $I\subset S$ be a monomial ideal of $S$ and $m$ be the number of minimal monomial generators of $I$, then $\sdepth(I)\geq \min\big\{1,n-\lfloor\frac{m}{2}\rfloor\big\}.$
\end{Theorem}
\section[Depth and Stanley depth of cyclic modules associated to $\mathcal{P}_{n,m}$ and $\mathcal{C}_{n,m}$ when $m\leq 3$]{Depth and Stanley depth of cyclic modules associated to $\mathcal{P}_{n,m}$ and $\mathcal{C}_{n,m}$ when $1\leq m\leq 3$ }
Let $n\geq 2$ and $1\leq i\leq n$, for convenience we take $x_i:=x_{i1}$, $y_i:=x_{i2}$ and $z_i:=x_{i3}$, see Figures \ref{A11} and \ref{A12}. We set $S_{n,1}:=K[x_1,x_2,\dots,x_n]$,
$S_{n,2}:=K[x_1,x_2,\dots x_n,y_1,y_2,\dots,y_n]$ and $S_{n,3}:=K[x_1,x_2,\dots x_n,y_1,y_2,\dots,y_n,z_1,z_2,\dots,z_n]$. Clearly $\mathcal{P}_{n,1}\cong P_n$ and $\mathcal{C}_{n,1}\cong C_n$, the minimal sets of monomial generators of the edge ideals of $\mathcal{P}_{n,2}$, $\mathcal{P}_{n,3}$, $\mathcal{C}_{n,2}$ and $\mathcal{C}_{n,3}$ are given as:
$$\mathcal{G}(I(\mathcal{P}_{n,2}))=\cup^{n-1}_{i=1}\{x_{i}y_{i},x_{i}y_{i+1},x_{i}x_{i+1},x_{i+1}y_{i},y_{i}y_{i+1}\}\cup\{x_{n}y_{n}\},$$
$$\mathcal{G}(I(\mathcal{P}_{n,3}))=\cup^{n-1}_{i=1}\{x_{i}y_{i},x_{i}y_{i+1},x_{i}x_{i+1},x_{i+1}y_{i},y_{i}y_{i+1},y_{i}z_{i},
y_{i}z_{i+1},y_{i+1}z_{i},z_{i}z_{i+1}\}\cup \{x_{n}y_{n},y_{n}z_{n}\},$$
$$\mathcal{G}(I(\mathcal{C}_{n,2}))=\mathcal{G}(I(\mathcal{P}_{n,2}))\cup\big\{x_{1}y_{n},x_{1}x_{n},y_{1}x_{n},y_{1}y_{n}\big\} \text{\,\,and}$$
$$\mathcal{G}(I(\mathcal{C}_{n,3}))=\mathcal{G}(I(\mathcal{P}_{n,3}))\cup\big\{x_{1}y_{n},x_{1}x_{n},y_{1}x_{n},y_{1}y_{n},y_{1}z_{n},z_{1}y_{n},
z_{1}z_{n}\}.$$
In this section, we compute depth and Stanley depth of the cyclic modules $S_{n,m}/I(\mathcal{P}_{n,m})$ and $S_{n,m}/I(\mathcal{C}_{n,m})$, when $m=1,2,3$.
\begin{figure}[h!]
\centering
%\vspace{.002cm}
  %\Requires \usepackage{graphicx}
 \includegraphics[width=8cm]{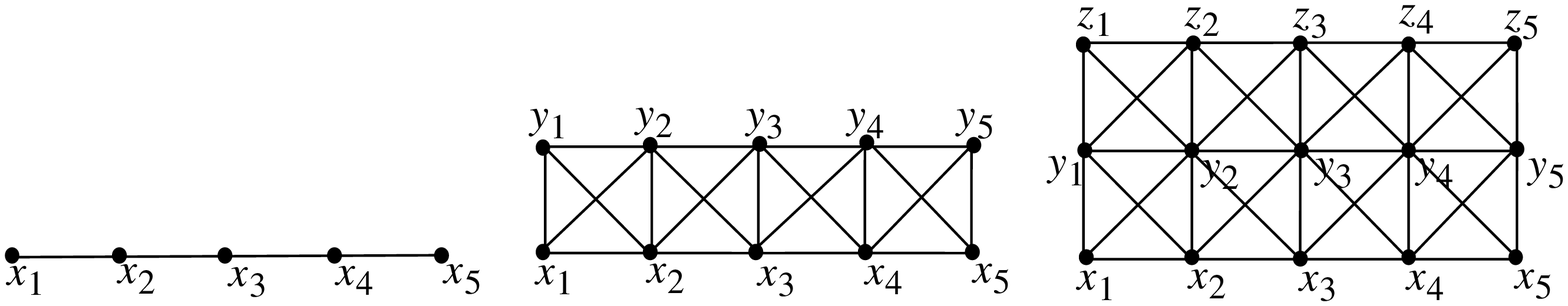}\\
 \caption {From left to right; $\mathcal{P}_{5,1}$, $\mathcal{P}_{5,2}$ and $\mathcal{P}_{5,3}$.}\label{A11}
\end{figure}

\begin{figure}[h!]
\centering
%\vspace{.002cm}
  %\Requires \usepackage{graphicx}
 \includegraphics[width=8cm]{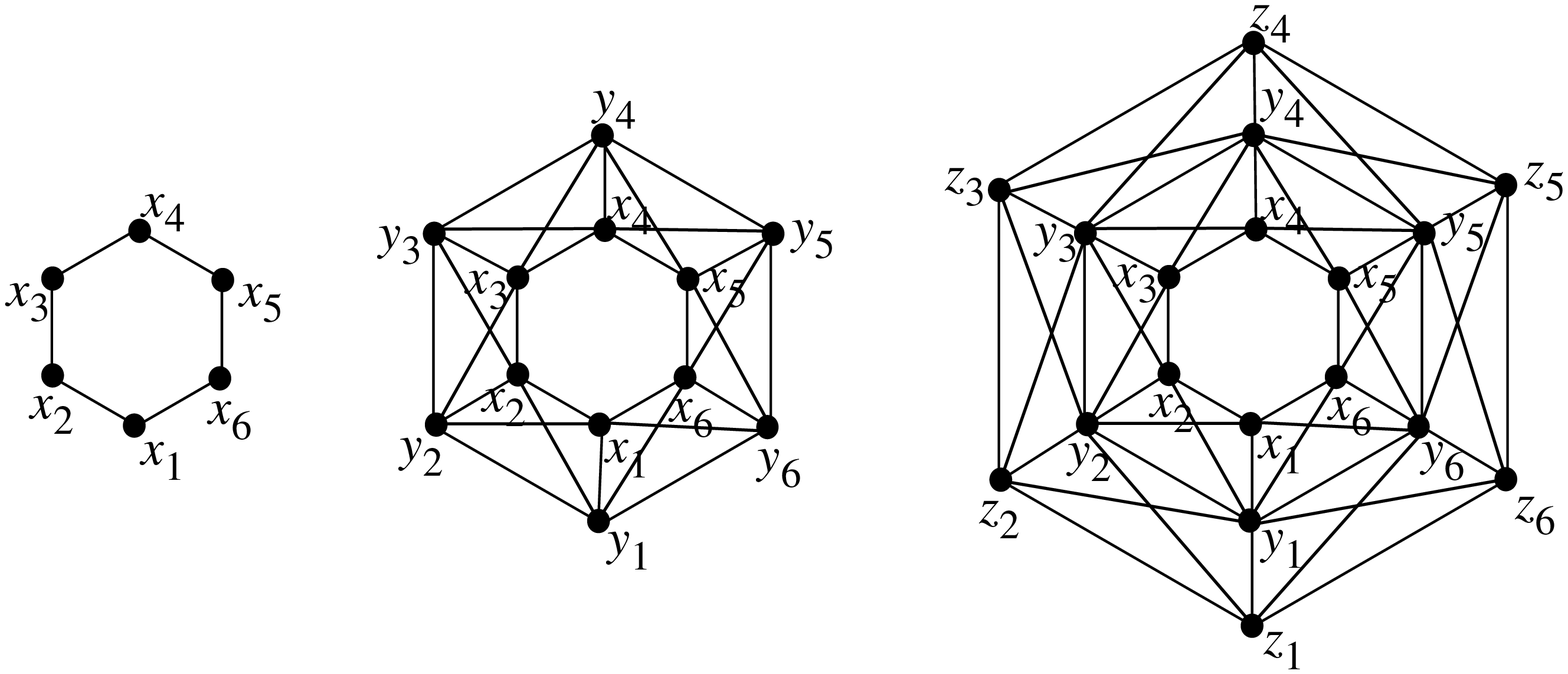}\\
 \caption {From left to right; $\mathcal{C}_{6,1}$, $\mathcal{C}_{6,2}$ and $\mathcal{C}_{6,3}$.}\label{A12}
\end{figure}

\begin{Remark}\label{case1}
{\em Note that for $n\geq 2$, $S_{n,1}/I(\mathcal{P}_{n,1})\cong S/I({P}_{n})$, thus by \cite[Lemma 2.8]{SM1} and \cite[Lemma 4]{ST} $\depth(S_{n,1}/I(\mathcal{P}_{n,1}))=\sdepth(S_{n,1}/I(\mathcal{P}_{n,1}))=\lceil\frac{n}{3}\rceil$.
Let $n\geq 3$, then $S_{n,1}/I(\mathcal{C}_{n,1})\cong S/I(C_n)$, and by \cite[Propositions 1.3,1.8]{MC4} $\depth(S_{n,1}/I(\mathcal{C}_{n,1}))=\lceil\frac{n-1}{3}\rceil\leq \sdepth(S_{n,1}/I(\mathcal{C}_{n,1}))\leq\lceil\frac{n}{3}\rceil.$}
\end{Remark}
\begin{Lemma}\label{The5}
Let $n\geq 1$, then $\depth(S_{n,2}/I(\mathcal{P}_{n,2}))=\sdepth(S_{n,2}/I(\mathcal{P}_{n,2}))=\lceil\frac{n}{3}\rceil$.
\end{Lemma}
\begin{proof}
If $n=1$ then by Remark \ref{case1} the result holds. Let $n\geq 2$, first we prove the result for $\depth$. Since $\diam(\mathcal{P}_{n,2})=n-1$, thus by \cite[Theorem 3.1]{FM} $\depth(S_{n,2}/I(\mathcal{P}_{n,2}))\geq \lceil\frac{n}{3}\rceil$. Now we prove the reverse inequality. For $n=2,3$ the required inequality is trivial. Let $n\geq 4$, we prove the inequality by induction on $n$. Since $y_{n-1}\not\in I(\mathcal{P}_{n,2})$, thus by \cite[Corollary 1.3]{AR1}
$$\depth(S_{n,2}/I(\mathcal{P}_{n,2}))\leq \depth(S_{n,2}/(I(\mathcal{P}_{n,2}):y_{n-1})).$$ As we can see that $S_{n,2}/(I(\mathcal{P}_{n,2}):y_{n-1})\cong S_{n-3,2}/I(\mathcal{P}_{n-3,2})[y_{n-1}]$, therefore by induction and Lemma \ref{111}
$\depth(S_{n,2}/(I(\mathcal{P}_{n,2}):y_{n-1}))\leq \lceil\frac{n-3}{3}\rceil+1=\lceil\frac{n}{3}\rceil$.
Proof for Stanley depth is similar using \cite[Theorem 4.18]{FM} and \cite[Proposition 2.7]{MC}.
 \end{proof}

\begin{Lemma}\label{Th31}
Let $n\geq 1$, then $\depth(S_{n,3}/I(\mathcal{P}_{n,3}))=\sdepth(S_{n,3}/I(\mathcal{P}_{n,3}))=\lceil\frac{n}{3}\rceil.$
\end{Lemma}
\begin{proof}
If $n=1$ then the result follows by Remark \ref{case1}. If $n=2$, then $S_{2,3}/I(P_{2,3})\cong S_{3,2}/I(P_{3,2})$ so we are done by Lemma \ref{The5}.  Let $n\geq 3$, we first prove the result for $\depth$. As $\diam(\mathcal{P}_{n,3})=n-1$, then by \cite[Theorem 3.1]{FM} we have $\depth(S_{n,3}/I(\mathcal{P}_{n,3}))\geq\lceil\frac{n}{3}\rceil$.
Now we prove the inequality $\depth(S_{n,3}/I(\mathcal{P}_{n,3}))\leq \lceil\frac{n}{3}\rceil$. If $n=3$, then the required inequality is trivial. Let $n\geq 4$, we prove the inequality by induction on $n$. As $y_{2}\not\in I(\mathcal{P}_{n,3})$, thus by \cite[Corollary 1.3]{AR1}
$$\depth(S_{n,3}/I(\mathcal{P}_{n,3}))\leq \depth(S_{n,3}/(I(\mathcal{P}_{n,3}):y_{2})).$$ Since $S_{n,3}/(I(\mathcal{P}_{n,3}):y_{2})\cong S_{n-3,3}/I(\mathcal{P}_{n-3,3})[y_{2}].$ Therefore by induction and Lemma \ref{111}
$$\depth(S_{n,3}/(I(\mathcal{P}_{n,3}):y_{2}))\leq \lceil\frac{n-3}{3}\rceil+1=\lceil\frac{n}{3}\rceil.$$
Proof for Stanley depth is similar using \cite[Theorem 4.18]{FM} and \cite[Proposition 2.7]{MC}.
\end{proof}
\begin{Theorem}\label{The7}
Let $n\geq 3$, then $\sdepth(S_{n,2}/I(\mathcal{C}_{n,2}))\geq\depth(S_{n,2}/I(\mathcal{C}_{n,2}))=\lceil\frac{n-1}{3}\rceil$.
\end{Theorem}
\begin{proof}
We first prove that $\depth(S_{n,2}/I(\mathcal{C}_{n,2}))=\lceil\frac{n-1}{3}\rceil$. For $ n=3,4$ the result is trivial. Let $n\geq5$, consider the short exact sequence
\begin{equation}\label{es13}
0\longrightarrow S_{n,2}/(I(\mathcal{C}_{n,2}):x_{n})\xrightarrow{\cdot x_{n}} S_{n,2}/I(\mathcal{C}_{n,2})\longrightarrow S_{n,2}/(I(\mathcal{C}_{n,2}),x_{n})\longrightarrow 0,
\end{equation}
by Depth Lemma
\begin{equation*}
\depth(S_{n,2}/I(\mathcal{C}_{n,2}))\geq \min\{\depth(S_{n,2}/(I(\mathcal{C}_{n,2}):x_{n})), \depth(S_{n,2}/(I(\mathcal{C}_{n,2}),x_{n}))\}.
\end{equation*}
\begin{multline*}
(I(\mathcal{C}_{n,2}):x_{n})=\big(\cup^{n-3}_{i=2}\{x_{i}y_{i},x_{i}y_{i+1},x_{i}x_{i+1},x_{i+1}y_{i},y_{i}y_{i+1}\},x_{n-2}y_{n-2},x_{1},y_{1},x_{n-1},y_{n-1},y_{n}\big).
\end{multline*}
After renumbering the variables, we have
$S_{n,2}/(I(\mathcal{C}_{n,2}):x_{n})\cong S_{n-3,2}/I(\mathcal{P}_{n-3,2})[x_{n}].$ Thus by Lemmas \ref{The5} and \ref{111}
$\depth(S_{n,2}/(I(\mathcal{C}_{n,2}):x_{n}))=\lceil\frac{n-3}{3}\rceil+1=\lceil\frac{n}{3}\rceil.$ And let
 \begin{multline*}
J=(I(\mathcal{C}_{n,2}),x_{n})=\big(\cup^{n-2}_{i=1}\{x_{i}y_{i},x_{i}y_{i+1},x_{i}x_{i+1},x_{i+1}y_{i},y_{i}y_{i+1}\},x_{n-1}y_{n-1},x_{n},
x_{n-1}y_{n},y_{n-1}y_{n},\\y_{1}y_{n},x_{1}y_{n}\big)
=(I(\mathcal{P}_{n-1,2}),x_{n},x_{n-1}y_{n},y_{n-1}y_{n},y_{1}y_{n},x_{1}y_{n}).
\end{multline*}
Consider the following exact sequence
\begin{equation}\label{es14}
0\longrightarrow S_{n,2}/(J:y_{n})\xrightarrow{\cdot y_{n}} S_{n,2}/J\longrightarrow S_{n,2}/(J,y_{n})\longrightarrow 0,
\end{equation}
by Depth Lemma
\begin{equation*}
\depth(S_{n,2}/J)\geq \min\{\depth(S_{n,2}/(J:y_{n})), \depth(S_{n,2}/(J,y_{n}))\}.
\end{equation*}
As $(J,y_{n}) = (I(\mathcal{P}_{n-1,2}),x_{n},y_{n})$ and $S_{n,2}/(J,y_{n})\cong S_{n-1,2}/I(\mathcal{P}_{n-1,2}).$ Therefore by Lemma \ref{The5}
 $\depth(S_{n,2}/(J,y_{n}))=\lceil\frac{n-1}{3}\rceil.$
 Also
 \begin{eqnarray*}
 (J:y_{n})=\big(\cup^{n-3}_{i=2}\{x_{i}y_{i},x_{i}y_{i+1},x_{i}x_{i+1},x_{i+1}y_{i},y_{i}y_{i+1}\},x_{n-2}y_{n-2},x_{1},y_{1},x_{n-1},y_{n-1},x_{n}\big).
\end{eqnarray*}
After renumbering the variables, we get
$S_{n,2}/(J:y_{n})\cong S_{n-3,2}/I(\mathcal{P}_{n-3,2})[y_{n}].$ Therefore by Lemmas \ref{The5} and \ref{111}
$\depth(S_{n,2}/(I(\mathcal{C}_{n,2}):y_{n}))=\lceil\frac{n-3}{3}\rceil+1=\lceil\frac{n}{3}\rceil.$
If $n\equiv 0(\mod 3)$ or $n\equiv 2(\mod 3)$ then $\lceil\frac{n-1}{3}\rceil=\lceil\frac{n}{3}\rceil.$ By applying Depth Lemma on exact sequences (\ref{es13})  and (\ref{es14}), we have $\depth(S_{n,2}/I(\mathcal{C}_{n,2}))=\lceil\frac{n-1}{3}\rceil$, as required.
Now for $n\equiv 1(\mod 3)$, assume that $n\geq7$, then we have the following $S_{n,2}$-module isomorphism:
{\small\begin{multline*}
(I(\mathcal{C}_{n,2}):x_{n})/I(\mathcal{C}_{n,2})\cong  x_{1}\frac{K[x_{3},\dots,x_{n-1},y_{3},\dots,,y_{n-1}]}{(\bigcup^{n-2}_{i=3}\{x_{i}y_{i},x_{i}y_{i+1},x_{i}x_{i+1},x_{i+1}y_{i},y_{i}y_{i+1}\},x_{n-1}y_{n-1}\big)}[x_{1}] \\ \oplus y_{1}\frac{K[x_{3},\dots,x_{n-1},y_{3},\dots,,y_{n-1}]}{(\bigcup^{n-2}_{i=3}\{x_{i}y_{i},x_{i}y_{i+1},x_{i}x_{i+1},x_{i+1}y_{i},y_{i}y_{i+1}\},x_{n-1}y_{n-1}\big)}[y_{1}] \\ \oplus
y_{n}\frac{K[x_{2},\dots,x_{n-2},y_{2},\dots,,y_{n-2}]}{(\bigcup^{n-3}_{i=2}\{x_{i}y_{i},x_{i}y_{i+1},x_{i}x_{i+1},x_{i+1}y_{i},y_{i}y_{i+1}\},x_{n-2}y_{n-2}\big)}[y_{n}]  \\ \oplus
x_{n-1}\frac{K[x_{2},\dots,x_{n-3},y_{2},\dots,y_{n-3}]}{(\bigcup^{n-4}_{i=2}\{x_{i}y_{i},x_{i}y_{i+1},x_{i}x_{i+1},x_{i+1}y_{i},y_{i}y_{i+1}\},x_{n-3}y_{n-3}\big)}[x_{n-1}] \\ \oplus
y_{n-1}\frac{K[x_{2},\dots,x_{n-3},y_{2},\dots,y_{n-3}]}{(\bigcup^{n-4}_{i=2}\{x_{i}y_{i},x_{i}y_{i+1},x_{i}x_{i+1},x_{i+1}y_{i},y_{i}y_{i+1}\},x_{n-3}y_{n-3}\big)}[y_{n-1}].
\end{multline*}}
We can see that the first three summands are isomorphic to $S_{n-3,2}/I(\mathcal{P}_{n-3,2})[x_n]$ and last two summands are isomorphic to $S_{n-4,2}/I(\mathcal{P}_{n-4,2})[x_n].$
Thus by Lemmas \ref{The5} and \ref{111}, we have
$$\depth(I(\mathcal{C}_{n,2}):x_n)/I(\mathcal{C}_{n,2}))=\min\{\lceil\frac{n-3}{3}\rceil+1,\lceil\frac{n-4}{3}\rceil+1\}= \lceil\frac{n-1}{3}\rceil.$$
Now by using Depth Lemma on the following short exact sequence we get the required result.
$$0\longrightarrow (I(\mathcal{C}_{n,2}):x_{n})/{I(\mathcal{C}_{n,2})}\xrightarrow{\cdot x_{n}} S_{n,2}/I(\mathcal{C}_{n,2})\longrightarrow S_{n,2}/(I(\mathcal{C}_{n,2}):x_{n})\longrightarrow 0.$$
For Stanley depth the required result follows by applying Lemma \ref{le1} on the exact sequences (\ref{es13}) and (\ref{es14}).
\end{proof}
\begin{Corollary}\label{cor6}
Let $n\geq 3$, then $\lceil\frac{n-1}{3}\rceil\leq\sdepth(S_{n,2}/I(\mathcal{C}_{n,2}))\leq\lceil\frac{n}{3}\rceil$.
\end{Corollary}
\begin{proof}
$I(\mathcal{C}_{3,2})$ is a square free Veronese ideal, by  \cite[Theorem 1.1]{MC8} $\sdepth(S_{n,2}/I(\mathcal{C}_{n,2}))=1$. Let $n\geq 4$, by \cite[Proposition 2.7]{MC} $\sdepth(S_{n,2}/I(\mathcal{C}_{n,2}))\leq\sdepth(S_{n,2}/(I(\mathcal{C}_{n,2}):x_{n})).$ Since $S_{n,2}/(I(\mathcal{C}_{n,2}):x_{n})\cong S_{n-3,2}/I(\mathcal{P}_{n-3,2})[x_{n}].$
Using Lemmas \ref{The5} and \ref{111} $\sdepth(S_{n,2}/(I(\mathcal{C}_{n,2}):x_{n}))=\lceil\frac{n-3}{3}\rceil+1=\lceil\frac{n}{3}\rceil.$
\end{proof}
For $n\geq 2$ we define a supergraph of $\mathcal{P}_{n,3}$ denoted by $\mathcal{P}^{\star}_{n,3}$ with the set of vertices $V(\mathcal{P}^{\star}_{n,3}):=V(\mathcal{P}_{n,3})\cup\{z_{n+1}\}$ and edge set  $E(\mathcal{P}^{\star}_{n,3}):=E(\mathcal{P}_{n,3})\cup \{z_nz_{n+1},y_nz_{n+1}\}$. Also we define a supergraph of $\mathcal{P}^{\star}_{n,3}$ denoted by $\mathcal{P}^{\star\star}_{n,3}$ with the set of vertices $V(\mathcal{P}^{\star\star}_{n,3}):=V(\mathcal{P}^{\star}_{n,3})\cup\{z_{n+2}\}$ and edge set $E(\mathcal{P}^{\star\star}_{n,3}):=E(\mathcal{P}^{\star\star}_{n,3})\cup \{z_1z_{n+2}, y_1z_{n+2}\}$. For examples of $\mathcal{P}^{\star}_{n,m}$ and $\mathcal{P}^{\star\star}_{n,m}$ see Fig. \ref{A1}. Let $S^{\star}_{n,3}:=S_{n,3}[z_{n+1}]$ and $S^{\star\star}_{n,3}:=S_{n,3}[z_{n+1},z_{n+2}]$ then we have the following lemmas:
\begin{figure}[h!]
\centering
%\vspace{.002cm}
  %\Requires \usepackage{graphicx}
 \includegraphics[width=8cm]{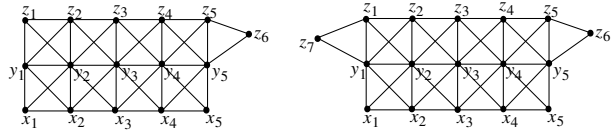}\\
 \caption {From left to right; $\mathcal{P}^{\star}_{5,3}$ and $\mathcal{P}^{\star\star}_{5,3}$.}\label{A1}
\end{figure}
\begin{Lemma}\label{Th32}
Let $n\geq 2$, then $\depth(S_{n,3}^{\star}/I(\mathcal{P}^{\star}_{n,3}))=\sdepth(S_{n,3}^{\star}/I(\mathcal{P}^{\star}_{n,3}))=\lceil\frac{n+1}{3}\rceil$.
\end{Lemma}
\begin{proof}
First we prove the result for depth.  Since $\diam(\mathcal{P}^{\star}_{n,3})=n$, then by \cite[Theorem 3.1]{FM} we have $\depth(S_{n,3}^{\star}/I(\mathcal{P}^{\star}_{n,3}))\geq\lceil\frac{n+1}{3}\rceil$.
Now we prove the reverse inequality, if $n=2$ then the result is trivial. Let $n\geq 3$, since $y_{n}\notin I(\mathcal{P}^{\star}_{n,3})$ so by \cite[Corollary 1.3]{AR1} $\depth(S_{n,3}^{\star}/I(\mathcal{P}^{\star}_{n,3}))\leq \depth(S_{n,3}^{\star}/(I(\mathcal{P}^{\star}_{n,3}):y_{n})).$
We have
$S_{n,3}^{\star}/(I(\mathcal{P}^{\star}_{n,3}):y_{n})\cong  (S_{n-2,3}/I(\mathcal{P}_{n-2,3}))[y_n].$
By Lemmas \ref{Th31} and \ref{111} $\depth(S_{n,3}^{\star}/(I(\mathcal{P}^{\star}_{n,3}):y_{n}))= \lceil\frac{n-2}{3}\rceil+1=\lceil\frac{n+1}{3}\rceil.$
Thus $\depth(S_{n,3}^{\star}/I(\mathcal{P}^{\star}_{n,3}))\leq\lceil\frac{n+1}{3}\rceil$. Proof for Stanley depth is similar using \cite[Proposition 2.7]{MC} and \cite[Theorem 4.18]{FM}.
\end{proof}

\begin{Lemma}\label{le33}
Let $n\geq 2$, then $\depth(S_{n,3}^{\star\star}/I(\mathcal{P}^{\star\star}_{n,3}))=\sdepth(S_{n,3}^{\star\star}/I(\mathcal{P}^{\star\star}_{n,3}))=\lceil\frac{n+2}{3}\rceil$.
\end{Lemma}
\begin{proof}
Clearly $\diam(\mathcal{P}^{\star\star}_{n,3})=n+1$, then by \cite[Theorem 3.1]{FM} we have $\depth(S_{n,3}^{\star\star}/I(\mathcal{P}^{\star\star}_{n,3}))\geq\lceil\frac{n+2}{3}\rceil$.
Now we prove the reverse inequality, the inequality is true when $n=2,3$. Let $n\geq 4$, as $y_{n}\notin I(\mathcal{P}^{\star\star}_{n,3})$ so by \cite[Corollary 1.3]{AR1} $\depth(S_{n,3}^{\star\star}/I(\mathcal{P}^{\star\star}_{n,3}))\leq \depth(S_{n,3}^{\star\star}/(I(\mathcal{P}^{\star\star}_{n,3}):y_{n})).$
Since $S_{n,3}^{\star\star}/(I(\mathcal{P}^{\star\star}_{n,3}):y_{n})\cong  (S^{\star}_{n-2,3}/I(\mathcal{P}^{\star}_{n-2,3}))[y_n].$
By Lemmas \ref{Th32} and \ref{111} we obtain $\depth(S_{n,3}^{\star}/I(\mathcal{P}^{\star}_{n,3}):y_{n})= \lceil\frac{n-2+1}{3}\rceil+1=\lceil\frac{n+2}{3}\rceil.$
Thus $\depth(S_{n,3}^{\star\star}/I(\mathcal{P}^{\star\star}_{n,3}))\leq\lceil\frac{n+2}{3}\rceil$.
Similarly one can prove the result for Stanley depth by using \cite[Proposition 2.7]{MC} and \cite[Theorem 4.18]{FM}.
\end{proof}
\begin{Theorem}\label{The4}
Let $n\geq 3$, if $n\equiv 0,2\,(\mod 3)$, then $\depth(S_{n,3}/I(\mathcal{C}_{n,3}))=\sdepth(S_{n,3}/I(\mathcal{C}_{n,3}))=\lceil\frac{n-1}{3}\rceil,$ and if $n\equiv 1\,(\mod 3)$, then $\lceil\frac{n-1}{3}\rceil\leq\depth(S_{n,3}/I(\mathcal{C}_{n,3})),\sdepth(S_{n,3}/I(\mathcal{C}_{n,3}))\leq \lceil\frac{n}{3}\rceil.$
\end{Theorem}
\begin{proof}
% We will prove this result by using Theorem \ref{Th31} and Lemma \ref{le33}.
We first prove the result for depth. For $n=3,4$ the result is clear. Let $n\geq5$, consider the short exact sequence
\begin{equation}\label{es41}
0\longrightarrow S_{n,3}/(I(\mathcal{C}_{n,3}):x_{n})\xrightarrow{\cdot x_{n}} S_{n,3}/I(\mathcal{C}_{n,3})\longrightarrow S_{n,3}/(I(\mathcal{C}_{n,3}),x_{n})\longrightarrow 0,
\end{equation}
\begin{multline*}
\text{Let\,\,}A:=(I(\mathcal{C}_{n,3}):x_{n})=\big(\cup^{n-3}_{i=2}\{x_{i}y_{i},x_{i}y_{i+1},x_{i}x_{i+1},x_{i+1}y_{i},y_{i}y_{i+1},y_{i}z_{i},y_{i}z_{i+1},y_{i+1}z_{i},
z_{i}z_{i+1}\},\\x_{n-2}y_{n-2},y_{n-2}z_{n-2},x_{1},y_{1},x_{n-1}
,y_{n-1},y_{n},z_{n}z_{n-1},z_{n-1}z_{n-2},y_{n-2}z_{n-1},z_{n}z_{1},z_{1}z_{2},y_{2}z_{1}\big),
\end{multline*}
and consider the following exact sequence
\begin{equation}\label{es42}
0\longrightarrow S_{n,3}/(A:z_{n})\xrightarrow{\cdot z_{n}} S_{n,3}/A\longrightarrow S_{n,3}/(A,z_{n})\longrightarrow 0,
\end{equation}
\begin{multline*}
(A,z_{n})=\big(\cup^{n-3}_{i=2}\{x_{i}y_{i},x_{i}y_{i+1},x_{i}x_{i+1},x_{i+1}y_{i},y_{i}y_{i+1},y_{i}z_{i},y_{i}z_{i+1},y_{i+1}z_{i},z_{i}z_{i+1}\},
x_{n-2}y_{n-2},\\y_{n-2}z_{n-2},x_{1},y_{1},x_{n-1},y_{n-1},y_{n},z_{n},z_{n-1}z_{n-2},y_{n-2}z_{n-1},z_{1}z_{2},y_{2}z_{1}\big).
\end{multline*}
After renumbering the variables, we have
%$$I(\mathcal{C}_n):x_{n} =(I(J_{n-3}),x_{1},y_{1},x_{n-1},y_{n-1},y_{n}).$$
%Thus
$S_{n,3}/(A,z_{n})\cong (S_{n-3,3}^{\star\star}/I(\mathcal{P}^{\star\star}_{n-3,3}))[x_{n}].$ Thus by Lemmas \ref{le33} and \ref{111}
$\depth(S_{n,3}/(A,z_{n}))=\lceil\frac{n-3+2}{3}\rceil+1=\lceil\frac{n-1}{3}\rceil+1.$
Also
\begin{multline*}
(A:z_{n})=\big(\cup^{n-3}_{i=2}\{x_{i}y_{i},x_{i}y_{i+1},x_{i}x_{i+1},x_{i+1}y_{i},y_{i}y_{i+1},y_{i}z_{i},y_{i}z_{i+1},y_{i+1}z_{i},z_{i}z_{i+1}\},
x_{n-2}y_{n-2},\\y_{n-2}z_{n-2},x_{1},y_{1},x_{n-1},y_{n-1},y_{n},z_{n-1},z_{1}\big).
\end{multline*}
After renumbering the variables, we get
%$$I(CJ_n):y_{n} =(I(J_{n-3}),x_{1},y_{1},x_{n-1},y_{n-1},x_{n}).$$
%Thus
$S_{n,3}/(A:z_{n})\cong (S_{n-3,3}/I(\mathcal{P}_{n-3,3}))[x_{n},z_{n}].$ Thus by Lemmas \ref{Th31} and \ref{111}
$\depth(S_{n,3}/(A:z_{n}))=\lceil\frac{n-3}{3}\rceil+2=\lceil\frac{n}{3}\rceil+1.$
Now let
 \begin{multline*}
\overline{A}:=(I(\mathcal{C}_{n,3}),x_{n})=\big(\cup^{n-2}_{i=1}\{x_{i}y_{i},x_{i}y_{i+1},x_{i}x_{i+1},x_{i+1}y_{i},y_{i}y_{i+1},y_{i}z_{i},y_{i}z_{i+1},
y_{i+1}z_{i},z_{i}z_{i+1}\},\\x_{n-1}y_{n-1},y_{n-1}z_{n-1},x_{n},x_{n-1}y_{n},y_{n-1}y_{n},y_{n}z_{n-1},y_{n-1}z_{n},z_{n-1}z_{n},y_{n}z_{n},y_{1}y_{n},
x_{1}y_{n},y_{1}z_{n},y_{n}z_{1},z_{1}z_{n}\big)\\=(I(\mathcal{P}_{n-1,3}),x_{n},x_{n-1}y_{n},y_{n-1}y_{n},y_{n}z_{n-1},y_{n-1}z_{n},z_{n-1}z_{n},
y_{n}z_{n},y_{1}y_{n},x_{1}y_{n},y_{1}z_{n},y_{n}z_{1},z_{1}z_{n}),
\end{multline*}
and the following exact sequence
\begin{equation}\label{es43}
0\longrightarrow S_{n,3}/(\overline{A}:y_{n})\xrightarrow{\cdot y_{n}} S_{n,3}/\overline{A}\longrightarrow S_{n,3}/(\overline{A},y_{n})\longrightarrow 0,
\end{equation}
\begin{multline*}
\text{As\,\,}(\overline{A}:y_{n})=\big(\cup^{n-3}_{i=2}\{x_{i}y_{i},x_{i}y_{i+1},x_{i}x_{i+1},x_{i+1}y_{i},y_{i}y_{i+1},y_{i}z_{i},y_{i}z_{i+1},y_{i+1}z_{i},z_{i}z_{i+1}\},x_{n-2}y_{n-2},\\y_{n-2}z_{n-2},x_{n}
,x_{1},y_{1},z_{1},x_{n-1},y_{n-1},z_{n-1},z_{n}\big).
\end{multline*}
After renumbering the variables, we get
%$$I(CJ_n):y_{n} =(I(J_{n-3}),x_{1},y_{1},x_{n-1},y_{n-1},x_{n}).$$
%Thus
$S_{n,3}/(\overline{A}:y_{n})\cong S_{n-3,3}/I(\mathcal{P}_{n-3,3})[y_{n}].$ Therefore by Lemmas \ref{Th31} and \ref{111}
$\depth(S_{n,3}/(\overline{A}:y_{n}))=\lceil\frac{n-3}{3}\rceil+1=\lceil\frac{n}{3}\rceil.$
Now let
 %\begin{equation*}
$${\widehat{A}}:=(\overline{A},y_{n})=(I(\mathcal{P}_{n-1,3}),x_{n},y_{n},y_{n-1}z_{n},z_{n-1}z_{n},y_{1}z_{n},z_{1}z_{n}),$$
%\end{equation*}
and the following short exact sequence
\begin{equation}\label{es44}
0\longrightarrow S_{n,3}/({\widehat{A}}:z_{n})\xrightarrow{\cdot z_{n}} S_{n,3}/{\widehat{A}}\longrightarrow S_{n,3}/({\widehat{A}},z_{n})\longrightarrow 0,
\end{equation}
thus
$S_{n,3}/({\widehat{A}},z_{n})\cong S_{n-1,3}/I(\mathcal{P}_{n-1,3}).$ Therefore by Lemma \ref{Th31} $\depth(S_{n,3}/({\widehat{A}},z_{n}))=\lceil\frac{n-1}{3}\rceil.$
Also
\begin{multline*}
({\widehat{A}}:z_{n})=\big(\cup^{n-3}_{i=2}\{x_{i}y_{i},x_{i}y_{i+1},x_{i}x_{i+1},x_{i+1}y_{i},y_{i}y_{i+1},y_{i}z_{i},y_{i}z_{i+1},y_{i+1}z_{i},z_{i}z_{i+1}\},
x_{n-2}y_{n-2},\\y_{n-2}z_{n-2},z_{1},y_{1},z_{n-1},y_{n-1},y_{n},x_{n},x_{n-1}x_{n-2},x_{n-1}y_{n-2},x_{1}x_{2},x_{1}y_{2}\big).
\end{multline*}
After renumbering the variables, we have
%$$I(\mathcal{C}_n):x_{n} =(I(J_{n-3}),x_{1},y_{1},x_{n-1},y_{n-1},y_{n}).$$
%Thus
$S_{n,3}/({\widehat{A}}:z_{n})\cong (S_{n-3,3}^{\star\star}/I(\mathcal{P}^{\star\star}_{n-3,3}))[z_{n}].$ Thus by Lemmas \ref{le33} and \ref{111}
$\depth(S_{n,3}/({\widehat{A}}:z_{n}))=\lceil\frac{n-3+2}{3}\rceil+1=\lceil\frac{n-1}{3}\rceil+1.$
By applying Depth Lemma on the exact sequences (\ref{es41}), (\ref{es42}), (\ref{es43}) and (\ref{es44}) we obtain $\depth(S_{n,3}/I(\mathcal{C}_{n,3}))\geq\lceil\frac{n-1}{3}\rceil$. For upper bound, by \cite[Corollary 1.3]{AR1} $\depth(S_{n,3}/I(\mathcal{C}_{n,3}))\leq \depth(S_{n,3}/(I(\mathcal{C}_{n,3}):y_{n})).$ Since $(S_{n,3}/(I(\mathcal{C}_{n,3}):y_{n}))\cong (S_{n-3,3}/(I(\mathcal{P}_{n-3,3}))[y_n]$, by Lemmas \ref{Th31} and \ref{111}  $\depth(S_{n,3}/I(\mathcal{C}_{n,3}))\leq\lceil\frac{n}{3}\rceil$,
if $n\equiv 0(\mod 3)$ or $n\equiv 2(\mod 3)$ then $\lceil\frac{n-1}{3}\rceil=\lceil\frac{n}{3}\rceil$.
If $n\equiv 1(\mod 3)$ then $\lceil\frac{n-1}{3}\rceil\leq\depth(S_{n,3}/I(\mathcal{C}_{n,3}))\leq \lceil\frac{n}{3}\rceil.$ Proof for Stanley depth is similar using Lemma \ref{le1} and \cite[Proposition 2.7]{MC}.
\end{proof}
\begin{Example}
{\em One can expect that $\depth(S_{n,3}/I(\mathcal{C}_{n,3}))=\lceil\frac{n-1}{3}\rceil$ as we have in \cite[Proposition 1.3]{MC4} and Theorem \ref{The7}. But examples show that in the essential case when $n\equiv 1(\mod 3)$ the upper bound in Theorem \ref{The4} is reached. For instance, when $n=4$, then $\depth(S_{4,3}/I(\mathcal{C}_{4,3}))=2=\lceil\frac{4}{3}\rceil.$}
\end{Example}
\section[Lower bounds for Stanley depth of $I(\mathcal{P}_{n,m})$ and $I(\mathcal{C}_{n,m})$ when $m\leq 3$]{Lower bounds for Stanley depth of $I(\mathcal{P}_{n,m})$ and $I(\mathcal{C}_{n,m})$ when $1\leq m\leq 3$}
In this section, we give some lower bounds for Stanley depth of $I(\mathcal{P}_{n,m})$ and $I(\mathcal{C}_{n,m})$, when $m\leq 3$. These bounds together with the results of previous section allow us to give a positive answer to the conjecture \ref{C1}. We begin this section with the following useful lemma:
\begin{Lemma}\label{lecon}
Let $A$ and $B$ be two disjoint sets of variables, $I_1\subset K[A]$ and $I_2\subset K[B]$ be square free monomial ideals such that $\sdepth_{K[A]}(I_1)> \sdepth(K[A]/I_1)$. Then $$\sdepth_{K[A\cup B]}(I_1+I_2)\geq \sdepth(K[A]/I_1)+\sdepth_{K[B]}(I_2).$$
\begin{proof}
Proof follows by \cite[Theorem 1.3]{MC}.
\end{proof}
\end{Lemma}

\begin{Remark}\label{path1}
{\em Since $I(\mathcal{P}_{n,1})\cong I(P_n)$, thus by \cite[Theorem 2.3]{O} and \cite[Prposition 2.1]{PFY} we have
$\sdepth(I(\mathcal{P}_{n,1}))>\sdepth(S_{n,1}/I(\mathcal{P}_{n,1}))=\lceil\frac{n}{3}\rceil.$}
\end{Remark}

\begin{Theorem}\label{th3}
Let $n\geq 1$, then $\sdepth (I(\mathcal{P}_{n,2}))> \sdepth(S_{n,2}/I(\mathcal{P}_{n,2}))=\lceil\frac{n}{3}\rceil.$
\end{Theorem}
\begin{proof}
Let $1\leq t \leq n$, then by Lemma \ref{The5} we have $\sdepth (S_{t,2}/I(\mathcal{P}_{t,2}))=\lceil\frac{t}{3}\rceil$. We use Lemma \ref{The5} in the proof without referring it again and again.  By the same lemma it is enough to show that $\sdepth (I(\mathcal{P}_{n,2}))> \lceil\frac{n}{3}\rceil$. The proof is by induction on $n$. If $n=1$ then by Remark \ref{path1} the required result follows. If $n=2,3$, then by \cite[Lemma 2.1]{KS}, $\sdepth (I(\mathcal{P}_{n,2}))>\lceil\frac{n}{3}\rceil$.
Now assume that $n\ge 4$. Since $x_{n-1}\not\in I(\mathcal{P}_{n,2})$, thus we have $$I(\mathcal{P}_{n,2})=I(\mathcal{P}_{n,2})\cap S'\oplus x_{n-1}\big(I(\mathcal{P}_{n,2}):x_{n-1}\big)S_{n,2},$$
where $S'=K[x_{1},x_{2},\dots,x_{n-2},x_n,y_{1},y_{2},\dots,y_{n}]$. Now $$I(\mathcal{P}_{n,2})\cap S'=\big(\mathcal{G}(I(\mathcal{P}_{n-2,2})),x_{n-2}y_{n-1},y_{n-2}y_{n-1},x_{n}y_{n},y_{n-1}x_{n},y_{n-1}y_{n}\big) \text{ and }$$ $$ \big(I(\mathcal{P}_{n,2}):x_{n-1}\big)S_{n,2}=\big(\mathcal{G}(I(\mathcal{P}_{n-3,2})),x_{n-2},y_{n-2},y_{n-1},x_{n},y_{n}\big)S_{n,2}.$$
As $y_{n-1}\not\in I(\mathcal{P}_{n,2})\cap S'$, so we get
\begin{eqnarray*}
I(\mathcal{P}_{n,2})\cap S'= (I(\mathcal{P}_{n,2})\cap S')\cap S''\oplus y_{n-1}\big(I(\mathcal{P}_{n,2})\cap S':y_{n-1}\big)S',
\end{eqnarray*}
where $S''=K[x_{1},\dots,x_{n-2},x_n,y_{1},\dots,y_{n-2},y_{n}]$. Thus
$$I(\mathcal{P}_{n,2})=(I(\mathcal{P}_{n,2})\cap S')\cap S''\oplus y_{n-1}\big(I(\mathcal{P}_{n,2})\cap S':y_{n-1}\big)S'\oplus x_{n-1}\big(I(\mathcal{P}_{n,2}):x_{n-1}\big)S_{n,2},$$
where
$$(I(\mathcal{P}_{n,2})\cap S')\cap S''=(\mathcal{G}(I(\mathcal{P}_{n-2,2})),x_ny_n)S''$$ and
$$(I(\mathcal{P}_{n,2})\cap S':y_{n-1}\big)S'=\big(\mathcal{G}(I(\mathcal{P}_{n-3,2})),x_{n-2},y_{n-2},x_{n},y_{n}\big)S'.$$
%$$\big(I(\mathcal{P}_{n,2}):x_{n-1}\big)=\big(\mathcal{G}(I(\mathcal{P}_{n-3,2})),x_{n-2},y_{n-2},y_{n-1},x_{n},y_{n}\big)S_{n,2}$$
By induction on $n$ and  Lemma \ref{lecon} we have
$$\sdepth((I(\mathcal{P}_{n,2})\cap S')\cap S'')\geq \sdepth(S_{n-2,2}/I(\mathcal{P}_{n-2,2}))+\sdepth_{K[x_n,y_n]}(x_ny_n).$$
Again by induction on $n$, Lemma \ref{lecon} and Lemma \ref{111} we have
$$\sdepth((I(\mathcal{P}_{n,2})\cap S':y_{n-1}\big)S')\geq\sdepth(S_{n-3,2}/I(\mathcal{P}_{n-3,2}))+\sdepth_T(x_{n-2},y_{n-2},x_{n},y_{n})+1$$
and $$\sdepth\big(\big (I(\mathcal{P}_{n,2}):x_{n-1}\big)S_{n,2}\big)\geq \sdepth(S_{n-3,2}/I(\mathcal{P}_{n-3,2}))+\sdepth_R(x_{n-2},y_{n-2},y_{n-1},x_{n},y_{n}\big)+1,$$
where $T=[x_{n-2},y_{n-2},x_{n},y_{n}]$ and $R=K[x_{n-2},y_{n-2},y_{n-1},x_{n},y_{n}]$. Thus $\sdepth((I(\mathcal{P}_{n,2})\cap S')\cap S'')>\lceil\frac{n}{3}\rceil$ as $\sdepth_{K[x_n,y_n]}(x_ny_n)=2$. By \cite[Theorem 2.2]{BH} we have
$\sdepth((I(\mathcal{P}_{n,2})\cap S':y_{n-1}\big)S')>\lceil\frac{n}{3}\rceil$ and $\sdepth(\big(I(\mathcal{P}_{n,2}):x_{n-1}\big)S_{n,2})> \lceil\frac{n}{3}\rceil$. This completes the proof.
\end{proof}
Now we introduce some notations for the case $m=3$. For $3\leq l\leq n-2$, let $J_l:=(x_{n-l},z_{n-l},x_{n-l+1},y_{n-l-1},z_{n-l+1},x_{n-l-1},z_{n-l-1})$, $I(P'_{l-1}):=(x_{n-l+2}x_{n-l+3},\dots,x_{n-1}x_n)$ and $I(P''_{l-1}):=(z_{n-l+2}z_{n-l+3},\dots,z_{n-1}z_n)$ be the monomial ideals of $S_{n,3}$. Consider the subsets of variables $D_{l}:=\{x_{n-l+2},x_{n-l+3},\dots,x_{n-1},x_n\}$, $D'_l:=\{z_{n-l+2},z_{n-l+3},\dots,z_{n-1},z_n\}$ and $D''_l:=\{x_{n-l},z_{n-l},x_{n-l+1},y_{n-l-1},z_{n-l+1},x_{n-l-1},z_{n-l-1}\}$. Let $L_l$ be a monomial ideal of $S_{n,3}$ such that $L_l=I(P'_{l-1})+I(P''_{l-1})+J_l$. With these notations we have the following lemma:
\begin{Lemma}\label{path}
Let $3\leq l\leq n-2$, then $\sdepth_{K[D_l\cup D'_l\cup D''_l]}(L_l)\geq \lceil\frac{l+2}{3}\rceil+1.$
\end{Lemma}
\begin{proof}
Since $L_l=I(P'_{l-1})+I(P''_{l-1})+J_l$, by \cite[Theorem 1.3]{MC}, we have
\begin{multline}\label{eq5}
$$\sdepth_{K[D_l\cup D'_l\cup D''_l]}(L_l)\geq \min\big\{\sdepth_{K[D_l\cup D'_l\cup D''_l]}(J_l),\min\{\sdepth_{K[D_l\cup D'_l]}(I(P'_{l-1})),\\\sdepth_{K[D_l]}(K[D_l]/I(P'_{l-1}))+\sdepth_{K[D'_l]}(I(P''_{l-1}))\}\big\}.$$
\end{multline}
By using \cite[Theorem 2.3]{O} and \cite[Proposition 2.1]{PFY}, Eq. \ref{eq5} implies that
\begin{multline*}\label{eq4}
$$\sdepth_{K[D_l\cup D'_l\cup D''_l]}(L_l)\geq \min \{4+2(l-2),\min\{2l-2-\lfloor\frac{l-2}{2}\rfloor, \lceil\frac{l-1}{3}\rceil+l-1-\lfloor\frac{l-2}{2}\rfloor\}\}\geq \lceil\frac{l+2}{3}\rceil+1.$$
\end{multline*}
\end{proof}

\begin{Theorem}\label{th11}
Let $n\geq 1$$, then \sdepth (I(\mathcal{P}_{n,3}))>\sdepth (S_{n,3}/I(\mathcal{P}_{n,3})).$
\end{Theorem}
\begin{proof}
Let $1\leq t \leq n$, then by Lemma \ref{Th31} we have $\sdepth (S_{t,3}/I(\mathcal{P}_{t,3}))=\lceil\frac{t}{3}\rceil$. We use Lemma \ref{Th31} in the proof several times without referring it. Using the same lemma it is enough to show that $\sdepth (I(\mathcal{P}_{n,3}))> \lceil\frac{n}{3}\rceil$. We proceed by induction on $n$. If $n=1$, then by Remark \ref{path1} the required result follows. If $n=2$, the result follows by Theorem \ref{th3}. If $n=3$ then by \cite[Lemma 2.1]{KS} $\sdepth (I(\mathcal{P}_{3,3}))>\lceil\frac{3}{3}\rceil$. If $n\geq 4$, then we consider the following decomposition of $I(\mathcal{P}_{n,3})$ as a vector space:
$$I(\mathcal{P}_{n,3})=I(\mathcal{P}_{n,3})\cap R_{1}\oplus y_{n}(I(\mathcal{P}_{n,3}):y_{n})S_{n,3}.$$ Similarly, we can decompose $I(\mathcal{P}_{n,3})\cap R_{1}$ by the following:
$$I(\mathcal{P}_{n,3})\cap R_{1}=I(\mathcal{P}_{n,3})\cap R_{2}\oplus y_{n-1}(I(\mathcal{P}_{n,3})\cap R_{1}:y_{n-1})R_{1}.$$
Continuing in the same way for $1\leq l\leq n-1$ we have
$$I(\mathcal{P}_{n,3})\cap R_{l}=I(\mathcal{P}_{n,3})\cap R_{l+1}\oplus y_{n-l}(I(\mathcal{P}_{n,3})\cap R_{l}:y_{n-l})R_{l},$$
where $R_l:=K[x_1,x_2,\dots x_n,y_{1},y_{2},\dots,y_{n-l},z_1,z_2,\dots,z_n]$.
Finally we get the following decomposition of $I(\mathcal{P}_{n,3})$:
\begin{equation*}
I(\mathcal{P}_{n,3})=I(\mathcal{P}_{n,3})\cap R_{n}\oplus\oplus_{l=1}^{n-1}y_{n-l}(I(\mathcal{P}_{n,3})\cap R_{l}:y_{n-l})R_{l}\oplus y_{n}(I(\mathcal{P}_{n,3}):y_{n})S_{n,3}.
\end{equation*}
Therefore
\begin{multline}\label{eq3}
\sdepth(I(\mathcal{P}_{n,3}))\geq \min\big\{\sdepth(I(\mathcal{P}_{n,3})\cap R_{n}),\sdepth((I(\mathcal{P}_{n,3}):y_{n})S_{n,3}),\\\min_{l=1}^{n-1}\{\sdepth((I(\mathcal{P}_{n,3})\cap R_{l}:y_{n-l})R_{l})\}\big\}.
\end{multline}
Since $$I(\mathcal{P}_{n,3})\cap R_{n}=\big((x_1x_2,x_2x_3,\dots,x_{n-1}x_n)+(z_1z_2,z_2z_3,\dots,z_{n-1}z_n)\big)K[x_1,\dots,x_n,z_1,\dots,z_n],$$
thus by \cite[Theorem 1.3]{MC} and \cite[Proposition 2.1]{PFY} we have $\sdepth(I(\mathcal{P}_{n,3})\cap R_{n})>\lceil\frac{n}{3}\rceil.$
As we can see that
$$(I(\mathcal{P}_{n,3}):y_{n})S_{n,3}=(\mathcal{G}(I(\mathcal{P}_{n-2,3}))+(x_n,z_n,x_{n-1},z_{n-1},y_{n-1}))[y_{n}].$$ Let $B:=K[x_n,z_n,x_{n-1},z_{n-1},y_{n-1}]$ thus by induction on $n$, Lemmas \ref{lecon} and \ref{111}
$$\sdepth((I(\mathcal{P}_{n,3}):y_{n})S_{n,3})>\sdepth(S_{n-2,3}/I(\mathcal{P}_{n-2,3}))+\sdepth_{B} (x_n,z_n,x_{n-1},z_{n-1},y_{n-1})+1.$$
By \cite[Theorem 2.2]{BH} we have $\sdepth((I(\mathcal{P}_{n,3}):y_{n})S_{n,3})> \lceil\frac{n}{3}\rceil.$
\begin{description}
\item [(1)] If $l=1$, then
$(I(\mathcal{P}_{n,3})\cap R_{1}:y_{n-1})R_{1}= \Big(\mathcal{G}(I(\mathcal{P}_{n-3,3}))+ J_1\Big)[y_{n-1}]$,
where  $J_1:=(x_{n-1},z_{n-1},x_{n},y_{n-2},z_{n},x_{n-2},z_{n-2})$, then by induction on $n$, Lemmas \ref{lecon} and \ref{111}, we have
$$\sdepth((I(\mathcal{P}_{n,3})\cap R_{1}:y_{n-1})R_{1})> \sdepth(S_{n-3,3}/I(\mathcal{P}_{n-3,3}))+\sdepth_{K[\supp(J_1)]} (J_1)+1,$$
by \cite[Theorem 2.2]{BH} we have $\sdepth((I(\mathcal{P}_{n,3})\cap R_{1}:y_{n-1})R_{1})>\lceil\frac{n}{3}\rceil.$
\item [(2)] If $l=2$ and $n\neq 4$, then
$$(I(\mathcal{P}_{n,3})\cap R_{2}:y_{n-2})R_{2}= \Big(\mathcal{G}(I(\mathcal{P}_{n-4,3}))+J_2\Big)[y_{n-2},x_{n},z_{n}],$$
where $J_2:=(x_{n-2},z_{n-2},x_{n-1},z_{n-1},x_{n-3},y_{n-3},z_{n-3})$, using the same arguments as in case(1) we have
$\sdepth((I(\mathcal{P}_{n,3})\cap R_{2}:y_{n-2})R_{2})>\lceil\frac{n}{3}\rceil.$
\item[(3)] If $3\leq l\leq n-3$, then $(I(\mathcal{P}_{n,3})\cap R_{l}:y_{n-l})R_{l}= \Big(\mathcal{G}(I(\mathcal{P}_{n-(l+2),3}))+\mathcal{G}(L_l)\Big)[y_{n-l}]$,
by induction on $n$, Lemmas \ref{lecon} and \ref{111}, we have
\begin{equation}\label{eq1}
\sdepth((I(\mathcal{P}_{n,3})\cap R_{l}:y_{n-l})R_{l})> \sdepth(S_{{n-(l+2),3}}/(I(\mathcal{P}_{n-(l+2),3})))+\sdepth_{K[D_l\cup D'_l\cup D''_l]}(L_l)+1,
\end{equation}
By Eq. \ref{eq1} and Lemma \ref{path} we have
$$\sdepth((I(\mathcal{P}_{n,3})\cap R_{l}:y_{n-l})R_{l})> \lceil\frac{n-(l+2)}{3}\rceil+\lceil\frac{l+2}{3}\rceil+1+1>\lceil\frac{n}{3}\rceil.$$
\item[(4)] If $l=n-2$, then
$(I(\mathcal{P}_{n,3})\cap R_{n-2}:y_{2})R_{n-2}= (\mathcal{G}(L_{n-2}))[y_{2}]$,  by Lemmas \ref{path} and \ref{111}
we have $\sdepth((I(\mathcal{P}_{n,3})\cap R_{n-2}:y_{2})R_{n-2})>\lceil\frac{n}{3}\rceil.$
\item [(5)] If $l=n-1$, then
$$(I(\mathcal{P}_{n,3})\cap R_{n-1}:y_{1})R_{n-1}= \big(I(P'_{n-2})+I(P''_{n-2})+J_{n-1}\big)K[D_{n-1}\cup D'_{n-1}\cup D''_{n-1}\cup \{y_{1}\}],$$
where $\mathcal{G}(J_{n-1})=\{x_1,z_1,x_2,z_2\}$, $D_{n-1}=\{x_3,x_4,\dots,x_n\}$, $D'_{n-1}=\{z_3,z_4,\dots,z_n\}$ and $D''_{n-1}=\{x_1,z_1,x_2,z_2\}$.  Using the proof of Lemma \ref{path} and by Lemma \ref{111} $$\sdepth_{K[D_{n-1}\cup D'_{n-1}\cup D''_{n-1}\cup \{y_{1}\}]} \big(I(P'_{n-2})+I(P''_{n-2})+J_{n-1}\big)>\lceil\frac{n}{3}\rceil,$$ that is
 $\sdepth((I(\mathcal{P}_{n,3})\cap R_{n-1}:y_{1})R_{n-1})>\lceil\frac{n}{3}\rceil.$
\end{description}
Thus by Eq. \ref{eq3} we get $\sdepth(I(\mathcal{P}_{n,3}))>\lceil\frac{n}{3}\rceil$.
\end{proof}

\begin{Proposition}\label{Pro6}
 Let $n\geq 3$, then $\sdepth({I(\mathcal{C}_{n,2})}/{I(\mathcal{P}_{n,2})})\geq\lceil\frac{n+2}{3}\rceil.$
\end{Proposition}
\begin{proof}
For $3 \leq n\leq 5$, we use \cite{HVZ} to show that there exist Stanley decompositions of desired Stanley depth.
When  $n=3$ or 4, then
\begin{eqnarray*}
I(\mathcal{C}_{n,2})/I(\mathcal{P}_{n,2})=x_{1}x_{n}K[x_{1},x_{n}]\oplus x_{1}y_{n}K[x_{1},y_{n}]\oplus y_{1}x_{n}K[y_{1},x_{n}]\oplus y_{1}y_{n}K[y_{1},y_{n}].
\end{eqnarray*}
If $n=5$, then
\begin{eqnarray*}
I(\mathcal{C}_{5,2})/I(\mathcal{P}_{5,2})=x_{1}x_{5}K[x_{1},x_{3},x_{5}]\oplus x_{1}y_{5}K[x_{1},x_{3},y_{5}]\oplus y_{1}x_{5}K[y_{1},x_{3},x_{5}]\oplus y_{1}y_{5}K[y_{1},x_{3},y_{5}]\\ \oplus x_{1}y_{3}x_{5}K[x_{1},y_{3},x_{5}]\oplus x_{1}y_{3}y_{5}K[x_{1},y_{3},y_{5}]\oplus y_{1}y_{3}y_{5}K[y_{1},y_{3},y_{5}]\oplus y_{1}y_{3}x_{5}K[y_{1},y_{3},x_{5}].
\end{eqnarray*}
Let $n\geq6$ and $T:=(\bigcup^{n-3}_{i=3}\{x_{i}y_{i},x_{i}y_{i+1},x_{i}x_{i+1},x_{i+1}y_{i},y_{i}y_{i+1}\},x_{n-2}y_{n-2}\big)\subset \tilde{S}$, where $\tilde{S}:=K[x_{3},x_{4},\dots,x_{n-2},y_{3},y_{4}\dots,y_{n-2}]$. Then we have the following $K$-vector space isomorphism:
$$I(\mathcal{C}_{n,2})/I(\mathcal{P}_{n,2})\cong   x_{1}x_{n}\frac{\tilde{S}}{T}[x_{1},x_{n}]  \oplus y_{1}y_{n}\frac{\tilde{S}}{T}[y_{1},y_{n}]  \oplus
x_{1}y_{n}\frac{\tilde{S}}{T}[x_{1},y_{n}]  \oplus
y_{1}x_{n}\frac{\tilde{S}}{T}[y_{1},x_{n}].$$
Thus by Lemmas \ref{The5} and \ref{111}, we have $\sdepth(I(\mathcal{C}_{n,2})/I(\mathcal{P}_{n,2}))\geq \lceil\frac{n+2}{3}\rceil.$
\end{proof}
For $n\geq 6$, let $Q=\{x_{1},y_{1},x_{2},y_{2},x_{n},y_{n},x_{n-1},y_{n-1}\}.$ Consider a subgraph $\mathcal{C}^{\diamond}_{n,3}$ of $\mathcal{C}_{n,3}$ with vertex set $V(\mathcal{C}^{\diamond}_{n,3})=V(\mathcal{C}_{n,3})\setminus Q$ and edge set
$$E(\mathcal{C}^{\diamond}_{n,3})=E(\mathcal{C}_{n,3})\setminus\{e\in E(\mathcal{C}_{n,3}): \text{ where $e$ has at least one end vertex in $Q$}\}.$$
For example of $\mathcal{C}^{\diamond}_{n,3}$ see Fig. \ref{C}.
\begin{figure}[h!]
\centering
%\vspace{.002cm}
  %\Requires \usepackage{graphicx}
 \includegraphics[width=4cm]{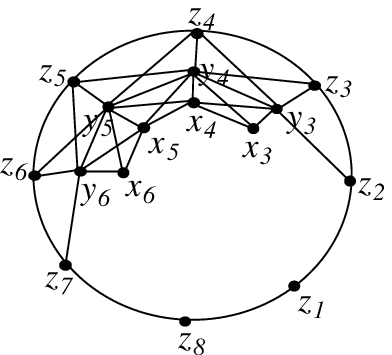}\\
 \caption {$\mathcal{C}^{\diamond}_{8,3}.$}\label{C}
\end{figure}

\begin{Lemma}\label{The425} Let $n\geq 6$, if $n\equiv 0\,(\mod 3)$, then $\sdepth(S^{\diamond}_{n,3}/I(C^{\diamond}_{n,3}))=\lceil\frac{n-2}{3}\rceil.$
Otherwise,\\ $\lceil\frac{n-2}{3}\rceil\leq\sdepth(S^{\diamond}_{n,3}/I(C^{\diamond}_{n,3}))\leq \lceil\frac{n}{3}\rceil.$
\end{Lemma}
\begin{proof}
Consider the short exact sequence
\begin{equation}\label{es10}
0\longrightarrow S^{\diamond}_{n,3}/(I(C^{\diamond}_{n,3}):z_{1})\xrightarrow{\cdot z_{1}} S^{\diamond}_{n,3}/I(C^{\diamond}_{n,3})\longrightarrow S^{\diamond}_{n,3}/(I(C^{\diamond}_{n,3}),z_{1})\longrightarrow 0,
\end{equation}
by Lemma \ref{le1}
\begin{equation*}
\sdepth(S^{\diamond}_{n,3}/I(C^{\diamond}_{n,3}))\geq \min\{\sdepth(S^{\diamond}_{n,3}/(I(C^{\diamond}_{n,3}):z_{1})), \sdepth(S^{\diamond}_{n,3}/(I(C^{\diamond}_{n,3}),z_{1}))\}.
\end{equation*}
\begin{multline*}
\text{As\,\,}(I(C^{\diamond}_{n,3}):z_{1})=((\cup^{n-3}_{i=3}\{x_{i}y_{i},x_{i}y_{i+1},x_{i}x_{i+1},x_{i+1}y_{i},y_{i}y_{i+1},y_{i}z_{i},y_{i}z_{i+1},y_{i+1}z_{i},z_{i}z_{i+1}\},\\x_{n-2}y_{n-2},y_{n-2}z_{n-2}),y_{n-2}z_{n-1},z_{n-2}z_{n-1},z_{2},z_{n}\big)
= (I(\mathcal{P}^{\star}_{n-4,3}),z_{2},z_{n}),
\end{multline*}
so we have
$S^{\diamond}_{n,3}/(I(C^{\diamond}_{n,3}):z_{1})\cong S^{\star}_{n-4,3}/I(\mathcal{P}^{\star}_{n-4,3})[z_{1}].$ Therefore, by Lemmas \ref{111} and \ref{Th32},
$$\sdepth(S^{\diamond}_{n,3}/(I(C^{\diamond}_{n,3}):z_{1}))=\lceil\frac{n-4+1}{3}\rceil+1=\lceil\frac{n}{3}\rceil.$$
Now suppose that
\begin{multline*}
$$B:=(I(C^{\diamond}_{n,3}),z_{1})=((\cup^{n-3}_{i=3}\{x_{i}y_{i},x_{i}y_{i+1},x_{i}x_{i+1},x_{i+1}y_{i},y_{i}y_{i+1},y_{i}z_{i},y_{i}z_{i+1},y_{i+1}z_{i},z_{i}z_{i+1}\},
\\x_{n-2}y_{n-2},y_{n-2}z_{n-2}),y_{n-2}z_{n-1},z_{n-2}z_{n-1},z_{n-1}z_{n},y_{3}z_{2},z_{2}z_{3},z_{1}\big)$$,
\end{multline*}
Applying Lemma \ref{le1} on the following short exact sequence
$$0\longrightarrow S^{\diamond}_{n,3}/(B:z_{n})\xrightarrow{\cdot z_{n}} S^{\diamond}_{n,3}/B\longrightarrow S^{\diamond}_{n,3}/(B,z_{n})\longrightarrow 0,$$
we have
$\sdepth(S^{\diamond}_{n,3}/B)\geq \min\{\sdepth(S^{\diamond}_{n,3}/(B:z_{n})), \sdepth(S^{\diamond}_{n,3}/(B,z_{n}))\}.$
\begin{multline*}
$$(B:z_{n})=((\cup^{n-3}_{i=3}\{x_{i}y_{i},x_{i}y_{i+1},x_{i}x_{i+1},x_{i+1}y_{i},y_{i}y_{i+1},y_{i}z_{i},y_{i}z_{i+1},y_{i+1}z_{i},z_{i}z_{i+1}\},
\\x_{n-2}y_{n-2},y_{n-2}z_{n-2}),y_{3}z_{2},z_{2}z_{3},z_{1},z_{n-1}\big)
= (I(\mathcal{P}^{\star}_{n-4,3}),z_{1},z_{n-1}),$$
 \end{multline*}
so we have
$S^{\diamond}_{n,3}/(B:z_{n})\cong S^{\star}_{n-4,3}/I(\mathcal{P}^{\star}_{n-4,3})[z_{n}].$ Therefore by Lemmas \ref{111} and \ref{Th32},
$\sdepth(S^{\diamond}_{n,3}/(B:z_{n}))=\lceil\frac{n-4+1}{3}\rceil+1=\lceil\frac{n}{3}\rceil.$
Now
\begin{multline*}
$$(B,z_{n})=((\bigcup^{n-3}_{i=3}\{x_{i}y_{i},x_{i}y_{i+1},x_{i}x_{i+1},x_{i+1}y_{i},y_{i}y_{i+1},y_{i}z_{i},y_{i}z_{i+1},y_{i+1}z_{i},z_{i}z_{i+1}\},
\\x_{n-2}y_{n-2},y_{n-2}z_{n-2}),y_{n-2}z_{n-1},z_{n-2}z_{n-1},y_{3}z_{2},z_{2}z_{3},z_{1},z_{n}\big)= (I(\mathcal{P}^{\star\star}_{n-4,3}),z_{1},z_{n}),$$
\end{multline*}
 thus we have
$S^{\diamond}_{n,3}/(B,z_{n})\cong S^{\star\star}_{n-4,3}/I(\mathcal{P}^{\star\star}_{n-4,3}).$ Therefore by Lemma \ref{le33}
$\sdepth(S^{\diamond}_{n,3}/(B,z_{n}))=\lceil\frac{n-4+2}{3}\rceil=\lceil\frac{n-2}{3}\rceil.$
For upper bound, as $z_{1}\notin I(C^{\diamond}_{n,3})$ so by \cite[Proposition 2.7]{MC} $$\sdepth(S^{\diamond}_{n,3}/I(C^{\diamond}_{n,3}))\leq \sdepth(S^{\diamond}_{n,3}/(I(C^{\diamond}_{n,3}):z_{1})).$$ Since $(S^{\diamond}_{n,3}/(I(C^{\diamond}_{n,3}):z_{1}))\cong (S^{\star}_{n-4,3}/I(\mathcal{P}^{\star}_{n-4,3}))[z_{1}]$. Thus by Lemmas \ref{111} and \ref{Th32},  $$\sdepth(S^{\diamond}_{n,3}/I(C^{\diamond}_{n,3}))\leq\lceil\frac{n}{3}\rceil,$$
if $n\equiv 0(\mod 3)$ then $\lceil\frac{n-2}{3}\rceil=\lceil\frac{n}{3}\rceil$.
If $n\equiv 1(\mod 3)$ or $n\equiv 2(\mod 3)$ then $$\lceil\frac{n-2}{3}\rceil\leq\sdepth(S^{\diamond}_{n,3}/I(C^{\diamond}_{n,3}))\leq \lceil\frac{n}{3}\rceil.$$
\end{proof}

\begin{Proposition}\label{Pro7}
 Let $n\geq 3$, then $\sdepth({I(\mathcal{C}_{n,3})}/{I(\mathcal{P}_{n,3})})\geq\lceil\frac{n+2}{3}\rceil.$
\end{Proposition}
\begin{proof}
For $3\leq n \leq 4,$ as the minimal generators of ${I(\mathcal{C}_{n,3})}/{I(\mathcal{P}_{n,3})}$ have degree 2, so by \cite[Lemma 2.1]{KS} $\sdepth ({I(\mathcal{C}_{n,3})}/{I(\mathcal{P}_{n,3})})\geq2=\lceil\frac{n+2}{3}\rceil.$
\text{If }$n=5$ then we use \cite{HVZ} to show that there exist Stanley decompositions of desired Stanley depth.
Let
\begin{eqnarray*}
H:=x_{1}x_{5}K[x_{1},x_{3},x_{5}]\oplus x_{1}y_{5}K[x_{1},x_{3},y_{5}]\oplus y_{1}x_{5}K[x_{3},x_{5},y_{1}]\oplus y_{1}y_{5}K[x_{3},y_{1},y_{5}]\\ \oplus z_{1}y_{5}K[x_{3},y_{5},z_{1}]\oplus z_{1}z_{5}K[z_{1},z_{3},z_{5}]\oplus y_{1}z_{5}K[y_{1},y_{3},z_{5}]
\end{eqnarray*}
Clearly, $H\subset {I(\mathcal{C}_{5,3})}/{I(\mathcal{P}_{5,3})}$. Let $v\in {I(\mathcal{C}_{5,3})}/{I(\mathcal{P}_{5,3})}$ be a sqaurefree monomial such that $v\notin H$ then $\deg (v)\geq 3$. Since
\begin{eqnarray*}\label{eq22}
 {I(\mathcal{C}_{5,3})}/{I(\mathcal{P}_{5,3})}=H\oplus_{v} v K[\supp(v)],
\end{eqnarray*}
Thus we have $\sdepth(I(\mathcal{C}_{5,3})/I(\mathcal{P}_{5,3}))\geq 3=\lceil\frac{5+2}{3}\rceil.$ Now for $n\geq6$, let $$U:=(\cup^{n-3}_{i=3}\{x_{i}y_{i},x_{i}y_{i+1},x_{i}x_{i+1},x_{i+1}y_{i},y_{i}y_{i+1},y_{i}z_{i},y_{i}z_{i+1},y_{i+1}z_{i},z_{i}z_{i+1}\},
x_{n-2}y_{n-2},y_{n-2}z_{n-2})$$ be a squarefree monomial ideal of $R:=K[x_3,\dots,x_{n-2},y_3,\dots,y_{n-2},z_3,\dots,z_{n-2}]$. Then we have the following $K$-vector space isomorphism:
{\begin{eqnarray*}
{I(\mathcal{C}_{n,3})}/{I(\mathcal{P}_{n,3})}\cong   y_{1}y_{n}\frac{R}{U}[y_{1},y_{n}]\oplus x_{1}y_{n}\frac{R[z_2]}{\big(\mathcal{G}(U),y_{3}z_{2},z_{2}z_{3}\big)}[x_{1},y_{n}]\oplus  z_{1}y_{n}\frac{R[x_2]}{\big(\mathcal{G}(U),y_{3}x_{2},x_{2}x_{3}\big)}[z_{1},y_{n}]\\\oplus
y_{1}x_{n}\frac{R[z_{n-1}]}{\big(\mathcal{G}(U),y_{n-2}z_{n-1},z_{n-2}z_{n-1}\big)}[y_{1},x_{n}]\oplus
y_{1}z_{n}\frac{R[x_{n-1}]}{\big(\mathcal{G}(U),y_{n-2}x_{n-1},x_{n-2}x_{n-1}\big)}[y_{1},z_{n}]\\\oplus
x_{1}x_{n}\frac{R[z_{1},z_{2},z_{n-1},z_{n}]}
{\big(\mathcal{G}(U),y_{n-2}z_{n-1},z_{n-2}z_{n-1},z_{n-1}z_{n},z_{n}z_{1},z_{1}z_{2},y_{3}z_{2},z_{2}z_{3}\big)}[x_{1},x_{n}]\\\oplus
z_{1}z_{n}\frac{R[x_{1},x_{2},x_{n-1},x_{n}]}
{\big(\mathcal{G}(U),y_{n-2}x_{n-1},x_{n-2}x_{n-1},x_{n-1}x_{n},x_{n}x_{1},x_{1}x_{2},y_{3}x_{2},x_{2}x_{3}\big)}[z_{1},z_{n}].
\end{eqnarray*}}
Clearly we can see that $R/U\cong S_{n-4,3}/I(\mathcal{P}_{n-4,3})$,
\begin{multline*}
\frac{R[z_2]}{\big(\mathcal{G}(U),y_{3}z_{2},z_{2}z_{3}\big)}\cong \frac{R[x_2]}{\big(\mathcal{G}(U),y_{3}x_{2},x_{2}x_{3}\big)}\cong
\frac{R[z_{n-1}]}{\big(\mathcal{G}(U),y_{n-2}z_{n-1},z_{n-2}z_{n-1}\big)} \\\cong \frac{R[x_{n-1}]}{\big(\mathcal{G}(U),y_{n-2}x_{n-1},x_{n-2}x_{n-1}\big)}\cong S^{\star}_{n-4,3}/I(\mathcal{P}^{\star}_{n-4,3}),
\end{multline*}
and
\begin{multline*}
\frac{R[z_{1},z_{2},z_{n-1},z_{n}]}
{\big(\mathcal{G}(U),y_{n-2}z_{n-1},z_{n-2}z_{n-1},z_{n-1}z_{n},z_{n}z_{1},z_{1}z_{2},y_{3}z_{2},z_{2}z_{3}\big)}\\\cong\frac{R[x_{1},x_{2},x_{n-1},x_{n}]}
{\big(\mathcal{G}(U),y_{n-2}x_{n-1},x_{n-2}x_{n-1},x_{n-1}x_{n},x_{n}x_{1},x_{1}x_{2},y_{3}x_{2},x_{2}x_{3}\big)}\cong S^{\diamond}_{n,3}/I(\mathcal{C}^{\diamond}_{n,3}).
\end{multline*}
Thus by Lemmas \ref{Th31}, \ref{Th32} and \ref{The425}, we have\\
$$\sdepth({I(\mathcal{C}_{n,3})}/{I(\mathcal{P}_{n,3})})\geq\min\Big\{\lceil\frac{n-4}{3}\rceil+2,\lceil\frac{n-4+1}{3}\rceil+2,
\lceil\frac{n-2}{3}\rceil+2\Big\}= \lceil\frac{n+2}{3}\rceil.$$
\end{proof}
\begin{Theorem}
Let $1\leq m\leq 3$, $n\geq 3$, then $\sdepth(I(\mathcal{C}_{n,m}))\geq \sdepth(S_{n,m}/I(\mathcal{C}_{n,m}))$.
\end{Theorem}
\begin{proof}
For $m=1$, $I(\mathcal{C}_{n,1})=C_n$. Then the result follows by \cite[Theorem 1.9]{MC4} and \cite[Theorem 2.3]{O}.
If $m=2 \text{\,\,or\,\,}3$, consider the short exact sequence $$0\longrightarrow I(\mathcal{P}_{n,m})\longrightarrow I(\mathcal{C}_{n,m})\longrightarrow I(\mathcal{C}_{n,m})/I(\mathcal{P}_{n,m})\longrightarrow 0,$$
then by Lemma \ref{le1}, $\sdepth (I(\mathcal{C}_{n,m}))\geq \min\{\sdepth (I(\mathcal{P}_{n,m})),\sdepth (I(\mathcal{C}_{n,m})/I(\mathcal{P}_{n,m}))\}.$
By Propositions \ref{th3} and \ref{th11}, we have $\sdepth (I(\mathcal{P}_{n,m}))\geq \lceil\frac{n}{3}\rceil+1,$ and by Propositions \ref{Pro6} and \ref{Pro7}, we have
$\sdepth(I(\mathcal{C}_{n,m})/I(\mathcal{P}_{n,m}))\geq \lceil\frac{n+2}{3}\rceil=\lceil\frac{n-1}{3}\rceil+1,$
this completes the proof.
\end{proof}

%\begin{Theorem}
%Let $n\geq 3$, then $\sdepth(I(\mathcal{C}_{n,3}))\geq \lceil\frac{n-1}{3}\rceil+1$.
%\end{Theorem}
%\begin{proof}
%Consider the short exact sequence $$0\longrightarrow I(\mathcal{P}_{n,3})\longrightarrow I(\mathcal{C}_{n,3})\longrightarrow I(\mathcal{C}_{n,3})/I(\mathcal{P}_{n,3})\longrightarrow 0,$$
%then by Lemma \ref{le1}, $$\sdepth (I(\mathcal{C}_{n,3}))\geq \min\{\sdepth (I(\mathcal{P}_{n,3})),\sdepth(I(\mathcal{C}_{n,3})/I(\mathcal{P}_{n,3}))\}.$$
%By Theorem \ref{th3}, we have $$\sdepth (I(\mathcal{P}_{n,3}))\geq \lceil\frac{n}{3}\rceil+1,$$ and by Proposition \ref{Pro7}, we have
%$$\sdepth(I(\mathcal{C}_{n,3})/I(\mathcal{P}_{n,3}))\geq \lceil\frac{n+2}{3}\rceil=\lceil\frac{n-1}{3}\rceil+1,$$
%this finishes the proof.
%\end{proof}
%\begin{Corollary}\label{Cor9}
%Let $n\geq 3$, then $$\sdepth (I(\mathcal{C}_{n,3})) \geq\sdepth(S_{n,3}/I(\mathcal{C}_{n,3}))+1.$$
%\end{Corollary}

\section[Upper bounds for depth and Stanley depth of cyclic modules associated to $\mathcal{P}_{n,m}$ and $\mathcal{C}_{n,m}$,]{Upper bounds for depth and Stanley depth of cyclic modules associated to $\mathcal{P}_{n,m}$ and $\mathcal{C}_{n,m}$}
Let $m\leq n$, in general we don't know the values of depth and Stanley depth of $S_{n,m}/I(\mathcal{P}_{n,m})$. However, in the light of our observations we propose the following open question.
\begin{Question}\label{Q1}
Is $\depth(S_{n,m}/I(\mathcal{P}_{n,m}))=\sdepth(S_{n,m}/I(\mathcal{P}_{n,m}))=\lceil\frac{n}{3}\rceil\lceil\frac{m}{3}\rceil?$
\end{Question}
Let $n\geq 2$, we have confirmed this question for the cases when $1\leq m\leq 3$ see Remark \ref{case1}, Lemma \ref{The5} and Lemma \ref{Th31}. If $m=4$, we make some calculations for depth and Stanley depth  by using CoCoA, (for sdepth we use SdepthLib:coc \cite{GR}). Calculations show that $\depth(S_{4,4}/I(\mathcal{P}_{4,4}))=\sdepth(S_{4,4}/I(\mathcal{P}_{4,4}))=4=\lceil\frac{4}{3}\rceil\lceil\frac{4}{3}\rceil,$ $\sdepth(S_{5,4}/I(\mathcal{P}_{5,4}))=4=\lceil\frac{5}{4}\rceil\lceil\frac{4}{3}\rceil,$ and  $\sdepth(S_{6,4}/I(\mathcal{P}_{6,4}))=4=\lceil\frac{6}{3}\rceil\lceil\frac{4}{3}\rceil.$ The following theorem gives a partial answer to the Question \ref{Q1}.

\begin{Theorem}\label{Th323}
Let $n\geq 2$, then $\depth(S_{n,m}/I(\mathcal{P}_{n,m})), \sdepth(S_{n,m}/I(\mathcal{P}_{n,m})) \leq\lceil\frac{n}{3}\rceil\lceil\frac{m}{3}\rceil.$ \end{Theorem}
\begin{proof}
Without loss of generality we can assume that $m\leq n$. We first prove the result for depth. When $m=1$, then $I(\mathcal{P}_{n,1})=I({P}_{n})$, we have the required result by Remark \ref{case1}. For $m=2,3$ the result follows from  Lemmas \ref{The5} and \ref{Th31}, respectively. Let $m\geq4$, we will prove this result by induction on $m$. Let $v$ be a monomial such that
$$v:=\left\{
  \begin{array}{ll}
    x_{2(m-1)}x_{5(m-1)}\dots x_{(n-4)(m-1)}x_{(n-1)(m-1)}, & \hbox{if  $n\equiv 0(\mod 3)$}; \\
    x_{1(m-1)}x_{4(m-1)}\dots x_{(n-3)(m-1)}x_{n(m-1)},& \hbox{if  $n\equiv 1(\mod 3)$}; \\
    x_{2(m-1)}x_{5(m-1)} \dots x_{(n-3)(m-1)}x_{n(m-1)}, & \hbox{if  $n\equiv 2(\mod 3)$}.
  \end{array}
\right.$$
clearly $v\notin I(\mathcal{P}_{n,m})$ so by \cite[Corollary 1.3]{AR1} $$\depth(S_{n,m}/I(\mathcal{P}_{n,m}))\leq \depth(S_{n,m}/(I(\mathcal{P}_{{n,m}}):v)).$$
%$(I(\mathcal{P}_{{n,4}}):v))=(I({P}_{n}),x_{1m},x_{2m},\dots,x_{nm},x_{1(m-2)},\\x_{2(m-2)}, \dots,x_{n(m-2)})$ and $S_{n,4}/(I(\mathcal{P}_{{n,4}}):v)\cong (S_{n,1}/I(\mathcal{P}_{n,1}))[supp(v)]$,
In all three cases $|\supp(v)|=\lceil\frac{n}{3}\rceil$ and $S_{n,m}/(I(\mathcal{P}_{{n,m}}):v)\cong (S_{n,m-3}/I(\mathcal{P}_{n,m-3}))[\supp(v)],$  so by induction and Lemma \ref{111}
$$\depth(S_{n,m}/I(\mathcal{P}_{{n,m}}))\leq\depth(S_{n,m}/(I(\mathcal{P}_{{n,m}}):v))\leq\lceil\frac{n}{3}\rceil\lceil\frac{m-3}{3}\rceil+\lceil\frac{n}{3}\rceil
=\lceil\frac{m}{3}\rceil\lceil\frac{n}{3}\rceil.$$
Similarly we can prove the result for sdepth by using \cite[Proposition 2.7]{MC}.
\end{proof}
\begin{Remark}
{\em For a positive answer to Question \ref{Q1} one needs to prove that $\lceil\frac{n}{3}\rceil\lceil\frac{m}{3}\rceil$ is a lower bound for depth and Stanley depth of $S_{n,m}/I(\mathcal{P}_{{n,m}})$. The lower bound $\lceil\frac{\diam(P_{n,m})+1}{3}\rceil$ (\cite[Theorems 3.1, 4.18]{FM}) which was helpful for the cases when $1\leq m\leq 3$ is no more useful if $m\geq4$. For instance,  $\depth(S_{4,4}/I(\mathcal{P}_{4,4}))=\sdepth(S_{4,4}/I(\mathcal{P}_{4,4}))=4$ but this lower bound shows that $\depth(S_{4,4}/I(\mathcal{P}_{4,4}))\geq 2=\lceil\frac{\diam(P_{4,4})+1}{3}\rceil$ and $\sdepth(S_{4,4}/I(\mathcal{P}_{4,4}))\geq 2=\lceil\frac{\diam(P_{4,4})+1}{3}\rceil$.
}
\end{Remark}
\begin{Theorem}\label{Th325}
Let $n\geq 3$ and $m\geq 1$, then
$$\depth(S_{n,m}/I(\mathcal{C}_{n,m}))\leq\left\{
  \begin{array}{ll}
\lceil\frac{n-1}{3}\rceil+(\lceil\frac{m}{3}\rceil-1)\lceil\frac{n}{3}\rceil, & \hbox{if $m\equiv 1,2(\mod 3)$;}\\
\lceil\frac{n}{3}\rceil\lceil\frac{m}{3}\rceil, & \hbox{if $m\equiv 0(\mod 3)$.}

  \end{array}
\right.$$
\end{Theorem}
\begin{proof}
We prove this result by induction on $m$. If $m=1$, then $I(\mathcal{C}_{n,1})=I({C}_{n})$, by \cite[Proposition 1.3]{MC4} we have the required result. For $m=2,3$ the result follows by Theorems \ref{The7} and \ref{The4}, respectively. Let $m\geq4$,
$$u:=\left\{
  \begin{array}{ll}
    x_{3(m-1)}x_{6(m-1)}\dots x_{(n-3)(m-1)}x_{n(m-1)}, & \hbox{if $n\equiv 0(\mod 3)$;} \\
    x_{1(m-1)}x_{4(m-1)}\dots x_{(n-6)(m-1)}x_{(n-3)(m-1)}x_{(n-1)(m-1)}, & \hbox{if $n\equiv 1(\mod 3)$;} \\
    x_{2(m-1)}x_{5(m-1)}\dots x_{(n-3)(m-1)}x_{n(m-1)}, & \hbox{if $n\equiv 2(\mod 3)$.}
  \end{array}
\right.$$
Clearly $u\notin I(\mathcal{C}_{n,m})$ and
$S_{n,m}/(I(\mathcal{C}_{{n,m}}):u)\cong (S_{n,m-3}/I(\mathcal{C}_{n,m-3}))[\supp(u)],$  since in all the cases $|\supp(u)|=\lceil\frac{n}{3}\rceil$, if $m\equiv 1,2(\mod 3)$ so by induction and Lemma \ref{111}
$$\depth(S_{n,m}/(I(\mathcal{C}_{{n,m}}):u))\leq\lceil\frac{n-1}{3}\rceil+(\lceil\frac{m-3}{3}\rceil-1)\lceil\frac{n}{3}\rceil+\lceil\frac{n}{3}\rceil=
\lceil\frac{n-1}{3}\rceil+(\lceil\frac{m}{3}\rceil-1)\lceil\frac{n}{3}\rceil.$$
Otherwise, by induction and Lemma \ref{111} we have
$$\depth(S_{n,m}/(I(\mathcal{C}_{{n,m}}):u))\leq\lceil\frac{n}{3}\rceil\lceil\frac{m-3}{3}\rceil+\lceil\frac{n}{3}\rceil=\lceil\frac{n}{3}
\rceil\lceil\frac{m}{3}\rceil.$$
\end{proof}
\begin{Theorem}\label{Th555}
Let $n\geq 3$ and $m\geq 1$, then
$\sdepth(S_{n,m}/I(\mathcal{C}_{n,m}))\leq \lceil\frac{n}{3}\rceil\lceil\frac{m}{3}\rceil.$
\end{Theorem}
\begin{proof}
The proof is similar to the proof of Theorem \ref{Th325} by using Corollary \ref{cor6} and Theorem \ref{The4} instead of Theorems \ref{The7} and \ref{The4}.
\end{proof}
\begin{Remark}
{\em The upper bounds for sdepth of $S_{n,m}/I(\mathcal{P}_{n,m})$ and $S_{n,m}/I(\mathcal{C}_{n,m})$ as proved in Theorems \ref{Th323} and \ref{Th555} are too sharp. On the bases of our observations we formulate the following open question. A positive answer to this question will prove the Conjecture \ref{C1}.}
\end{Remark}
\begin{Question}
Is $\sdepth(I(\mathcal{P}_{n,m})),\sdepth(I(\mathcal{C}_{n,m}))\geq \lceil\frac{n}{3}\rceil\lceil\frac{m}{3}\rceil?$
\end{Question}

\end{document}